\documentclass{amsart}
\usepackage{amssymb,amsmath,amsmath,latexsym,amssymb,amsfonts,amsbsy, amsthm,mathtools,graphicx,CJKutf8,CJKnumb,CJKulem,color}
\numberwithin{equation}{section}

\allowdisplaybreaks
\let\al=\alpha

\let\d=\delta
\let\e=\varepsilon

\let\s=\sigma
\let\f=\frac

\let\om=\omega

\let\D=\Delta

\let\tri=\triangle

\let\pa=\partial

\def\cS{{\mathcal S}}

\def\Im{\mathbf {Im}}

\def\R{\mathbb R}
\def\Z{ \mathbb Z}
\def\T{\mathbb T}

\def\eqdef{\buildrel\hbox{\footnotesize def}\over =}

\def\D{\langle D\rangle}
\def\k{\langle k\rangle}

\newcommand{\beq}{\begin{equation}}
\newcommand{\eeq}{\end{equation}}
\newcommand{\ben}{\begin{eqnarray}}
\newcommand{\een}{\end{eqnarray}}
\newcommand{\beno}{\begin{eqnarray*}}
\newcommand{\eeno}{\end{eqnarray*}}

\newtheorem{theorem}{Theorem}[section]

\newtheorem{lemma}[theorem]{Lemma}
\newtheorem{proposition}[theorem]{Proposition}
\newtheorem{corol}[theorem]{Corollary}
\newtheorem{remark}[theorem]{Remark}
\begin{document}
\begin{CJK*}{UTF8}{gkai}
\title[Gevrey stability of hydrostatic approximate]{Gevrey stability of hydrostatic approximate for the Navier-Stokes equations in a thin domain}

\author[C. Wang]{Chao Wang}
\address{School of Mathematical Sciences\\ Peking University\\ Beijing 100871,China}
\email{wangchao@math.pku.edu.cn}

\author[Y. Wang]{Yuxi Wang}
\address{School of Mathematical Sciences\\ Peking University\\ Beijing 100871,China}
\email{wangyuxi0422@pku.edu.cn}

\author[Z. Zhang]{Zhifei Zhang}
\address{School of Mathematical Sciences\\ Peking University\\ Beijing 100871,China}
\email{zfzhang@math.pku.edu.cn}

\date{\today}

\maketitle

\begin{abstract}
In this paper, we justify the limit from the Navier-Stokes system in a thin domain to the hydrostatic Navier-Stokes/Prandtl system {for the convex initial data with Gevrey 9/8 regualrity in $x$.}
\end{abstract}

\section{introduction}

In this paper, we consider 2-D incompressible Navier-Stokes equations in a thin domain when the depth of the domain and the viscosity  coefficient converge to zero simultaneously in a related way:
\begin{equation}\label{eq:NS}
\left\{\begin{array}{l}{\partial_{t} U+U \cdot \nabla U-\varepsilon^{2} (\pa_x^2+\eta\pa_y^2) U+\nabla P=0 \quad \text { in } \mathcal{S}^{\varepsilon} \times (0, \infty) }, \\ {\operatorname{div} U=0},\\
U|_{y=0}=U|_{y=\e}=0,
\end{array}\right.
\end{equation}
where $\mathcal{S}^\e=\left\{(x, y) \in \mathbb{T}\times\mathbb{R} : 0<y<\varepsilon\right\}$, and  $U(t,x,y), P(t,x,y)$ stand for the velocity and pressure function respectively and $\eta$ is a positive constant independent of $\varepsilon$. 
The system is prescribed with the initial data of the form
\begin{equation*}
\left.U\right|_{t=0}=\left(u_{0}\left(x, \frac{y}{\varepsilon}\right), \varepsilon v_{0}\left(x, \frac{y}{\varepsilon}\right)\right)=U_{0}^{\varepsilon} \quad \text {in}\quad \mathcal{S}^{\varepsilon}.
\end{equation*}

We rescale $(U, P)$ as follows
\begin{equation*}
U(t, x, y)=\left(u^{\varepsilon}\left(t, x, \frac{y}{\varepsilon}\right), \varepsilon v^{\varepsilon}\left(t, x, \frac{y}{\varepsilon}\right)\right) \quad \text { and } \quad P(t, x, y)=p^{\varepsilon}\left(t, x, \frac{y}{\varepsilon}\right).
\end{equation*}
Then the system \eqref{eq:NS}  is reduced to  the following scaled anisotropic Navier-Stokes system:
\begin{align}\label{eq:ANS}
\left\{
\begin{aligned}
&\pa_t u^\e+u^\e\pa_x u^\e+v^\e\pa_y u^\e-\e^2\pa_x^2 u^\e-\eta\pa_y^2 u^\e+\pa_x p^\e=0\quad \text { in } \mathcal{S} \times (0, \infty) ,\\
&\e^2(\pa_t v^\e+u^\e\pa_x v^\e+v^\e\pa_y v^\e-\e^2\pa_x^2 v^\e-\eta\pa_y^2 v^\e)+\pa_y p^\e=0\quad \text { in } \mathcal{S} \times (0, \infty) ,\\
&\pa_x u^\e+\pa_y v^\e=0 \quad \text { in } \mathcal{S} \times (0, \infty),\\
&(u^\e,v^\e)|_{y=0,1}=0,\\
&(u^\e,v^\e)|_{t=0}=(u_0,v_0) \quad \text { in  }\mathcal{S},
\end{aligned}
\right.
\end{align}
where $\mathcal{S}=\left\{(x, y) \in \mathbb{T}\times(0,1) \right\}$.
This is a classical model in geophysical fluid, where the vertical dimension of the domain is very small compared with the  horizontal  dimension of the domain. For simplicity, we take $\eta=1$ in the sequel and denote $\Delta_\e=\e^2\pa_x^2+\pa_y^2$.

Formally, taking $\e\to 0$ in \eqref{eq:ANS}, we derive the hydrostatic Navier-Stokes/Prandtl system (see \cite{LL, Renardy}):
\begin{align}\label{eq:HyNS}
\left\{
\begin{aligned}
&\pa_t u^p+u^p\pa_x u^p+v^p\pa_y u^p-\pa_y^2 u^p+\pa_x p^p=0\quad \text { in } \mathcal{S} \times (0, \infty) ,\\
&\pa_y p^p=0\quad \text { in } \mathcal{S} \times (0, \infty) ,\\
&\pa_x u^p+\pa_y v^p=0 \quad \text { in } \mathcal{S} \times (0, \infty),\\
&(u^p,v^p)|_{y=0,1}=0,\\
&u^p|_{t=0}=u_0 \quad \text { in }\,\, \mathcal{S}.
\end{aligned}
\right.
\end{align}

The goal of this paper is to justify the limit from the scaled anisotropic Navier-Stokes system \eqref{eq:ANS}  
to the hydrostatic Navier-Stokes system \eqref{eq:HyNS}. \smallskip

The first step is to deal with the well-posedness of  the system \eqref{eq:HyNS}. Similar to the classical Prandtl equation, nonlinear term $v^p\pa_y u^p$ will lead to one derivative loss in $x$ via the direct energy estimate to \eqref{eq:HyNS}. 
Indeed, the system \eqref{eq:HyNS}  may not be well-posed for general data in the Sobolev space \cite{Renardy}. 
However, the system is well-posed for analytic data \cite{PZZ}. A classical result for the Prandtl equation is the well-posedness in the Sobolev space for monotonic data in $y$ direction \cite{Olei, AW, MW}. This kind of data is forbidden  for the system \eqref{eq:HyNS} due to the boundary condition. Recently,  Gerard-Varet, Masmoudi and Vicol \cite{GMV} proved the well-posedness of the system \eqref{eq:HyNS}  for class of convex data in the Gevrey class $\f98$.

A natural question is whether the limit could be justified in the Gevrey class $\f 9 8$. In a recent work \cite{PZZ}, Paicu, Zhang and the third author justified the global (in time) limit for small analytic data. For the data in Gevrey class or Sobolev space, this question is highly nontrivial. In fact,  although the Prandtl equation is well-posed in the Sobolev space or Gevrey class(see \cite{GM, CWZ, LY, DG}), the question of the inviscid limit of the Navier-Stokes equations in the same spaces remains a challenging problem.\smallskip

Motivated by the methods introduced in \cite{GMV, WWZ}, we justify the limit from the system \eqref{eq:ANS}  
to  \eqref{eq:HyNS} for a class of convex data in the Gevrey class $\f 98$.
More precisely,   we consider the initial data of the form
\begin{align}\label{ass:data-form}
u^\e(0,x,y)=u_0(x,y),\quad v^\e(0,x,y)=v_0(x,y),
\end{align}
which  satisfy the {compatibility conditions }
\begin{align}
&\pa_x u_0+\pa_y  v_0=0,\quad u_0(x,0)=u_0(x,1)=v_0(x,0)=v_0(x,1)=0,\label{ass:data-com1} \\
&\int_0^1\pa_x u_0dy=0,\quad \pa_y^2 u_0|_{y=0,1}=\int_0^1(-\pa_x u^2_0+\pa_y^2u_0)dy-\int_\mathcal{S}\pa_y^2 u_0,\label{ass:data-com2}
\end{align}
and the convex condition
\begin{align}\label{ass:data-convex}
\inf_{\mathcal{S}}\pa_y^2u_0=2\delta_0>0.
\end{align}
We further assume that initial data falls into the Gevrey class with the bound
 \begin{align}\label{ass:data-bound}
 \| \pa_y u_0\|_{X^{N_0}_{\sigma,\tau_0}}+ \|  \pa_y^3 u_0\|_{X^{N_0-4}_{\sigma,\tau_0}}=M<+\infty.
\end{align}
Here the Gevrey class normal $\|\cdot\|_{X^{r}_{\sigma,\tau}}$ is defined by
\begin{align*}
\|f\|_{X^r_{\sigma,\tau}}^2=\| e^{\tau \langle D_x\rangle^{\sigma}}f\|_{H^{r,0}}^2,
 \end{align*}
with $\|f\|_{H^{r,s}}=\|\| f\|_{H^r_x(\mathbb{T})}\|_{H^s_y( 0,1)}.$\smallskip

For the data satisfying \eqref{ass:data-com1}-\eqref{ass:data-bound}, following the proof  in \cite{GMV}, one can prove the following local well-posedness result for the system \eqref{eq:HyNS}.

 \begin{theorem} \label{thm:HyNS}
Let the initial data $u_0$ satisfy  \eqref{ass:data-com1}-\eqref{ass:data-bound}  with $\sigma\in [\f89,1], \tau_0>0$ and $N_0\ge 10$.  Then there exist $T>0$ and a unique solution $u^p$ of \eqref{eq:HyNS}, which satisfies 
\begin{align*}
&\sup_{t\in[0,T]}\big(\|\pa_y u^p(t)\|_{X^{N_0-1}_{\sigma,\tau}}+\|\pa_y^3 u^p(t)\|_{X^{N_0-5}_{\sigma,\tau}}\big)<+\infty,\\
&\sup_{t\in[0,T]\times \mathcal{S}}\pa_y^2 u^p>\d_0.
\end{align*}
\end{theorem}

\smallskip

Now we state the main result of this paper.

\begin{theorem}\label{thm:limit}
Let initial data $u_0$ satisfies  \eqref{ass:data-com1}-\eqref{ass:data-bound}  with $\sigma\in [\f89,1], \tau_0>0$ and $N_0\ge 10$.  Then  there exists a unique solution of the Navier–Stokes equations \eqref{eq:ANS} in $[0,T]$, which satisfies
\begin{align*}
\|(u^\e-u^p,\e v^\e-\e v^p)\|_{L^2_{x,y}\cap L^\infty_{x,y}}\leq C\e^2,
\end{align*}
where $(u^p,  v^p)$ is given by Theorem \ref{thm:HyNS} and $C$ is a constant independent of $\e$. 
\end{theorem}

\begin{remark}
The range $\sigma\in [\f 89,1]$ should not be optimal. According to \cite{GGN}, the optimal range may be $[
\f 23,1]$. 
\end{remark}

\medskip

Let us sketch main ingredients of our proof and structure of this paper. \smallskip

\begin{itemize}

\item  {\bf Error equation}. In section 3, we introduce the error 
\begin{align*}
u^R=u^\e-u^p,\quad v^R=v^\e-v^p,\quad p^R=p^\e-p^p,
\end{align*}
which satisfy 
\begin{align*}
\left\{
\begin{aligned}
&\pa_t u^R-\Delta_\e u^R+\pa_x p^R+u^\e\pa_x u^R+u^R\pa_x u^p+v^\e\pa_y u^R+v^R\pa_y u^p-\e^2\pa_x^2 u^p=0,\\
&\e^2(\pa_t v^R-\Delta_\e v^R)+\pa_y p^R+\e^2(\pa_t v^p -\e^2 \pa_x^2 v^p-\pa_y^2 v^p+u^\e\pa_x v^\e+v^\e\pa_y v^\e)=0.
\end{aligned}
\right.
\end{align*}
The main difficulty comes from the term $v^R\pa_y u^p$, since $v^R$ is controlled via the relation $v^R=-\int_0^y\pa_xu^Rdy$,
which will lead to one derivative loss in $x$ variable. In \cite{PZZ}, the authors used the analyticity to overcome this difficulty.
For the data in the Gevrey class, we have to introduce new ideas.  

\item {\bf The vorticity formulation and hydrostatic trick}.
In \cite{GMV}, the authors introduced the vorticity formulation of \eqref{eq:HyNS}:  
\beno
\pa_t\om-\pa_y^2\om+u\pa_x\om+v\pa_y\om=0,\quad \om=\pa_y u.
\eeno
If we test $\om$ to this equation, then the term $v\pa_y\om$ still lose one derivative. 
In \cite{GMV}, the first key idea is to use the so called hydrostatic trick, i.e., test the vorticity equation by $\f {\om} {\pa_y\om}$,
which makes sense under the convex assumption. Indeed,  the trouble term vanishes due to 
\beno
\int_{\cS}v\pa_y\om \f \om {\pa_y\om}dxdy=\int_{\cS}v\om dxdy=0.
\eeno
However, the viscosity term $\pa_y^2\om$ will give rise to new difficulty due to $\om|_{y=0,1}\neq 0$.  The second key idea introduced in \cite{GMV} is to introduce the boundary corrector  $\om^b$ defined by
\beno
\pa_t\om^b-\pa_y^2\om^b=0,\quad \pa_y\om^b_{y=0,1}\approx -\pa_x\int^1_0u^2dy,
\eeno
and then use the hydrostatic trick for the equation of $\om^{in}=\om-\om^b$.

Motivated by \cite{GMV, Mae, WWZ}, we introduce the vorticity formulation of the error equations in section 3:
 \begin{align*}
\left\{
\begin{aligned}
&\pa_t \om^R-\tri_{\e}\om^R +f=N(\om^R, \om^R),\\
&(\pa_y+\e|D|) \om^R|_{y=0}=\pa_xh^0+\cdots,\\
&(\pa_y-\e|D|) \om^R|_{y=1}=\pa_xh^1+\cdots.
\end{aligned}
\right.
\end{align*}
Here $\pa_xh^0$ and $\pa_xh^1$ are the worst terms on the boundary. 
To handle them, we also introduce the boundary corrector in section 4:
\beno
\pa_t\om^{b,i}-\pa_y^2\om^{b,i}=0,\quad \pa_y\om^{b}|_{y=i}=\pa_xh^i.
\eeno
In section 5, we control the boundary corrector $\om^{bl}=\om^{b,0}+\om^{b,1}$ via the interior vorticity $\om^{in}=\om^R-\om^{bl}$.

\item {\bf Energy estimate for the interior vorticity}. In section 6, using the hydrostatic trick, we derive the following energy estimate:
\begin{align*}
&\sup_{s\in[0,t]}\|\om^{in}(s)\|_{X^r}^2+\int_{0}^t\|(\pa_y,\e\pa_x)\om^{in})\|_{X^r}^2ds+\beta\int_0^t\|\om^{in}\|_{X^{r+\f{\sigma}2}}^2ds\\
&\leq Ct \e^4+2\d\int_0^t\|\mathcal{N}\|_{X^{r-\f{\s}4}}^2ds+C\e^2\int_0^t\|P_{\geq N(\e)}(\pa_y,\e\pa_x)(u^R, \e v^R)\|_{X^{r+1}}^2ds+\cdots,
\end{align*}
where the third term on the right comes from the following boundary term in the energy estimate
\begin{align*}
\Big|\int_0^t\int_{\mathbb{T}}\e|D|\D^r\om^{R}_\Phi~\f{\D^r\om^{in}_\Phi}{\pa_y\om^p}|_{y=0,1}dxds\Big|,
\end{align*}
which is bounded by 
\begin{align*}
&\int_0^t\big(\|\e|D|\om^{in}\|_{X^{r}}+\|\e|D| \om^{bl}\|_{X^r}\big)\big(\|\pa_y\om^{in}\|_{X^{r}}+\|\om^{in}\|_{X^{r}}\big)\\
&\qquad+\big(\|\e|D|\om^{bl}\|_{X^{r-\f{\s}2}}+\|\e|D|\pa_y\om^{bl}\|_{X^{r-\f{\s}2}}\big)\|\om^{in}\|_{X^{r+\f{\s}2}}ds.
\end{align*}
New trouble is to control the  term $\int_0^t\|\e|D|\om^{in}\|_{X^{r}}^2dt$. For this, we need to make a high-low frequency decomposition for $\om^{in}$
so that 
\begin{align*}
&\int_0^t\|P_{\leq 2N(\e)}\e|D|\om^{in}\|_{X^{r}}^2ds\leq C\int_0^t\|\om^{in}\|_{X^{r+\f{\sigma}{2}}}^2ds
\end{align*}
and 
\begin{align*}
&\int_0^t\|P_{\geq N(\e)}\e|D|\om^{in}\|_{X^{r}}^2ds\leq C\e^2\int_0^t\|P_{\geq N(\e)}(\pa_y,\e\pa_x)(u^R, \e v^R)\|_{X^{r+1}}^2ds\\
&\qquad+\int_0^t\big(\|P_{\geq N(\e)}(\pa_y,\e\pa_x)(u^R, \e v^R)\|_{X^{r+1-\sigma}}^2+\|\om^{in}\|_{X^{r+\f{\sigma}2}}^2\big)ds.
\end{align*}
 This decomposition is the key observation of this paper, which is motivated by the fact that 
\beno
\|P_{\ge N(\e)}f\|_{X^r}\le C\|P_{\ge N(\e)}\e f\|_{X^{r+1-\f \sigma 2}},
\eeno
which is very useful for the control of $v^R$ instead of the usual control $\|v^R\|_{X^r}\le \|u^R\|_{X^{r+1}}$(losing one derivative).

\item {\bf Energy estimate for the velocity}. 
In section 7, we derive the following energy estimate:
\begin{align*}
&\e^2\|P_{\geq N(\e)}(u^R,\e v^R)(t)\|_{X^{r+1}}^2+\beta \e^2 \int_0^t\|P_{\geq N(\e)}(u^R,\e v^R)\|_{X^{r+1+\f{\s}{2}}}^2\\
&\qquad +\int_0^t\e^2\|P_{\geq N(\e)}(\pa_y,\e\pa_x)( u^{R},\e v^R)\|_{X^{r+1}}^2\\
&\leq  C\int_0^t \|\om^{in}\|_{X^{r+\f{\s}2}}^2ds+\d\int_0^t\|P_{\geq N(\e)}(\mathcal{N}_u,\e \mathcal{N}_v)\|_{X^{r+1-\f{\s}2}}^2ds.
\end{align*}

\item {\bf Nonlinear estimates and bootstrap argument}. In section 8, we make the nonlinear estimates for $(\mathcal{N}, \mathcal{N}_u,\e \mathcal{N}_v)$. Based on the energy estimates for $(\om^R, u^R, v^R)$ and nonlinear estimates, we close our energy estimates by using a standard bootstrap argument in section 9.
\end{itemize}

\medskip

Throughout this paper,  we denote by $C$ a constant independent of $\e,\beta$. We denote by {$N(\e)=[\e^{-\f 2 {2-\sigma}}]$} an integer.

\section{Gevrey class and Elliptic equation in a strip}

\subsection{Some estimates in Gevrey class}

Let us define
\begin{align}\label{def: tau}
f_\Phi=\mathcal{F}^{-1}(e^{\Phi(t,k)}\widehat{f}(k))=e^{\Phi(t,D)}f,\quad \Phi(t,k)\eqdef\tau(t)\langle k\rangle^{\sigma}.
\end{align}
 Obviously, for $\sigma\in [0,1]$ and $\tau(t)\ge 0$, $\Phi(t,k)$ satisfies the subadditive inequality
\begin{align}\label{eq:subadditive}
\Phi(t,k)\leq \Phi(t,k-\ell)+\Phi(t,\ell).
\end{align}
Then we have
\begin{align*}
\|f\|_{X^r_{\sigma,\tau}}=\|f_\Phi\|_{H^{r,0}}.
 \end{align*}
 For functions which only depend on variable $x$, we denote
\beno
|f|_{X^r_{\sigma,\tau}}=\|f_\Phi\|_{H^r_x(\mathbb{T})}.
\eeno
It is easy to see that if $r'\geq r,$ then $\|\cdot\|_{X^{r'}_{\sigma,\tau}}\geq \|\cdot\|_{X^{r}_{\sigma,\tau}}.$ 
For simplicity, we drop subscript $\sigma,\tau$ in the notations $\|f\|_{X^r_{\sigma,\tau}}, |f|_{X^r_{\sigma,\tau}}$ etc.
We say that  a function $f$ belongs to Gevrey class $\f 1 \sigma$ if $\|f\|_{X^r_{\sigma,\tau}}<+\infty$. When $\sigma=1$, the function
is analytic. In the sequel, we always take 
\beno
\tau(t)=\tau_0e^{-\beta t}, \quad \tau_0>0,\quad  \beta\ge 1(\text{to be determined later}).
\eeno

We introduce the frequency cut-off operators $P_{\geq N}$ and $P_{\leq N}$, which are defined by
\begin{align}\label{def: P_N}
P_{\geq N}f(x)=\f 1 {2\pi}\sum_{k\in \Z}\chi\big(\f{k}{N}\big)\widehat{f}(k)e^{ikx},\quad P_{\leq N}f=\big(1-P_{\geq N-1}\big)f.
\end{align}
Here  function $\chi $ is a  smooth even function with $\rm{supp}\chi\in[\f12,+\infty)\cup (-\infty,-\f12]$ and $\chi(x)=1$ for $|x|\geq1$.

\begin{lemma}\label{lem:com-S}
Let $r\ge 0,~s_1>\f32$, $s>\f12$ and $0\leq\delta\leq 1$. Then it holds that
\begin{align*}
&\big\|[\langle D\rangle^r,f]\pa_xg\big\|_{L^2_x}\leq C\|f\|_{H^{s_1}_x}\|g\|_{H^r_x}+C\|f\|_{H^{r+1-\d}_x}\|g\|_{H^{s+\d}_x},\\
&\big\|[P_{\geq N},f]\pa_xg\big\|_{H^r_x}\leq C\|f\|_{H^{s_1}_x}\|P_{\geq \f{N}{2}}g\|_{H^{r}_x}+C\|f\|_{H^{r+1-\d}_x}\|g\|_{H^{s+\d}_x}.
\end{align*}

\end{lemma}
\begin{proof}
The first inequality is classical(see \cite{BCD} for example). Here we only present the proof for the second one.

Thanks to Plancherel formula, we have
\begin{align*}
\big\|[P_{\geq N},f]\pa_xg\big\|_{H^r_x}=&\big\|\langle k\rangle^r\big(\chi(\f{k}{N})(\widehat{f}*\widehat{\pa_xg})-\widehat{f}*(\chi(\f{\cdot}{N})\widehat{\pa_xg}\big)\big\|_{\ell^2_k}\\
\leq&\big\|\langle k\rangle^r\sum_{\ell\in\Z}\big(\chi(\f{k}{N})-\chi(\f{\ell}{N})\big)|\ell||\widehat{f}(k-\ell)| |\widehat{g}(\ell)|\big\|_{\ell^2_k}.
\end{align*}
We consider two cases. For  $|\ell|\leq 2|k-\ell|$, we have
\begin{align*}
&\langle k\rangle^r|\ell|\leq C\langle \ell\rangle^\d\langle k-\ell\rangle^{r+1-\d},
\end{align*}
which implies that
\begin{align*}
\Big\|\langle k\rangle^r\sum_{|\ell|\le 2|k-\ell|}\big(\chi(\f{k}{N})-\chi(\f{\ell}{N})\big)|\ell||\widehat{f}(k-\ell)| |\widehat{g}(\ell)|\Big\|_{\ell^2_k}\leq C\|f\|_{H^{r+1-\d}_x}\|g\|_{H^{s+\d}_x}.
\end{align*}
Here we used $s>\f12$.
If  $|\ell|\geq 2|k-\ell|$, then we have
\begin{align*}
\f{|\ell|}{2}\leq |k|\leq \f{3|\ell|}{2},
\end{align*}
and $k,\ell$ must have the same sign. Using the mean value theorem, we get
\begin{align*}
\chi\big(\f{k}{N}\big)-\chi\big(\f{\ell}{N}\big)=\f{1}{N}\chi'\big(\f{\xi}{N}\big)(k-\ell),
\end{align*}
where $\xi$ is some point  between $k$ and $\ell.$ In this case, we have
\begin{align*}
|\ell| \Big|\chi(\f{k}{N})-\chi(\f{\ell}{N})\Big|\leq & \chi(\f{\ell}{N/2})\big|\f{\xi}{N}\chi'(\f{\xi}{N})\big||k-\ell|\f{|\ell|}{|\xi|}\\
\leq& C\chi(\f{\ell}{N/2})|k-\ell|.
\end{align*}
Therefore, we obtain
\begin{align*}
\Big\|\langle k&\rangle^r\sum_{|\ell|\ge 2|k-\ell|}(\chi(\f{k}{N})-\chi(\f{\ell}{N}))|\ell||\widehat{f}(k-\ell)| |\widehat{g}(\ell)|\Big\|_{\ell^2_k}\\
\leq &\Big\|\sum_{\ell\in\mathbb{Z}}|k-\ell||\widehat{f}(k-\ell)\langle\ell\rangle^r |\chi(\f{\ell}{N/2})|\widehat{g}(\ell)|\Big\|_{\ell_k^2}\\
\leq& C\|f\|_{H^{s_1}_x}\|P_{\geq \f{N}{2}}g\|_{H^{r}_x}.
\end{align*}
Here we used $s_1>\f32.$ This shows the second inequality. 
\end{proof}

\begin{lemma}\label{lem:product-Gev}
Let $r\ge 0$ and $s> \f12$. Then  for any $N\in \Z_+$,
\beno
&& |fg|_{X^r}\leq C  |f|_{X^{s}}   |g |_{X^r}  +  C |f |_{X^r}  |g |_{X^{s}},\\
 &&|P_{\geq N}(f g)  |_{X^r}\leq C|f|_{X^s}  |P_{\geq \f{N}{2}}g |_{X^r}+C|P_{\geq \f{N}{2}}f |_{X^r}  |g |_{X^{s}}.
 \eeno
\end{lemma}

\begin{proof}
By Plancherel formula, we have
\begin{align*}
|fg|_{X^r}=\big\|\langle k\rangle^re^{\Phi(t,k)}(\widehat{f}*\widehat{g})(k)\big\|_{\ell^2_k}.
\end{align*}
We get by \eqref{eq:subadditive} that
\begin{align*}
\Big|e^{\Phi(t,k)}(\widehat{f}*\widehat{g})(k)\Big|\leq& \sum_{\ell\in \mathbb{Z}}e^{\Phi(t,k-\ell)}|\widehat{f}|(k-\ell)e^{\Phi(t,\ell)}|\widehat{g}|(\ell)\\
=&\sum_{\ell\in \mathbb{Z}}\widehat{f^{+}_\Phi}(k-\ell) \widehat{g^{+}_{\Phi}}(\ell)=\widehat{f^{+}_\Phi }*\widehat{g^{+}_{\Phi}}=\mathcal{F}(f^{+}_\Phi g^{+}_{\Phi}),
\end{align*}
where we denote $f^{+}$ by  the Fourier transformation inverse of $|\widehat{f}|$ and so does $g^{+}$.
Then by the classical product estimate in Sobolev space, we obtain
\begin{align*}
|fg|_{X^r}\leq \|f^{+}_\Phi g^{+}_{\Phi}\|_{H^r_x}\leq& C\big(\|f^{+}_\Phi\|_{H^s_x}\|g^{+}_\Phi\|_{H^r_x}+\|f^{+}_\Phi\|_{H^r_x}\|f^{+}_\Phi\|_{H^s_x}\big)\\
=&C\big(|f|_{X^{s}}   |g |_{X^r} + |f |_{X^r}  |g |_{X^{s}}\big),
\end{align*}
where we use the fact that map $f_\Phi\mapsto
f^+_\Phi$ preserves the $L^2_x$ norm and $s>\f12$.

For the second one, we use the fact that if $2|k-\ell|\geq \ell$, then
\beno
\langle k\rangle^r\leq C\langle k-\ell\rangle^r,\quad \chi(\f{k}{N})\leq\chi(\f{k-\ell}{N/2}),
\eeno
and if $|\ell|\geq 2|k-\ell|$, then 
\beno
\langle k\rangle^r\leq C_r\langle \ell\rangle^r,\quad \chi(\f{k}{N})\leq \chi(\f{\ell}{N/2}).
\eeno
Then as in the argument for the first inequality, we have
\begin{align*}
|P_{\geq N}(f g)  |_{X^r}\le& C\|g^{+}_\Phi P_{\geq\f{N}{2}}\langle D_x\rangle^r f^{+}_\Phi\|_{L^2}+C\|f^{+}_\Phi P_{\geq\f{N}{2}}\langle D_x\rangle^r g^{+}_\Phi\|_{L^2}\\
\le&C\big(|f|_{X^s}  |P_{\geq \f{N}{2}}g |_{X^r}+|P_{\geq \f{N}{2}}f |_{X^r}  |g |_{X^{s}}\big).
\end{align*}
This shows the second inequality. 
\end{proof}

\begin{lemma}\label{lem:com-Gev}
Let $r\ge 0,~s_1>\f32$, $s>\f12$ and $0\leq\delta\leq 1$. Then it holds that
\begin{align*}
&\|(f\pa_xg)_{\Phi}-f\pa_xg_{\Phi}\|_{H^r_x}\leq C|f|_{X^{s_1}}|g|_{X^{r+\sigma}}+C|f|_{X^{r+1-\delta}}|g|_{X^{s+\delta}},\\
&\|P_{\geq N}(f\pa_xg)_{\Phi}-f(P_{\geq N}\pa_xg_{\Phi})\|_{H^r_x}\leq C|f|_{X^{s_1}}|P_{\geq \f{N}{2}}g|_{X^{r+\sigma}}+C|f|_{X^{r+1-\delta}}|g|_{X^{s+\delta}},
\end{align*}
for any $N\in\mathbb{Z}_+.$
\end{lemma}
\begin{proof}
By Plancherel formula and the argument in Lemma \ref{lem:product-Gev}, we have
\begin{align*}
\|(f\pa_xg)_{\Phi}-f\pa_xg_{\Phi}\|_{H^r_x}
\leq&\Big\|\langle k\rangle^r\sum_{\ell\in \mathbb{Z}}m(k,\ell)|\ell|\widehat{f^{+}_\Phi}(k-\ell) \widehat{g^{+}_{\Phi}}(\ell)\Big\|_{\ell^2_k},
\end{align*}
where $m(k,\ell)=(e^{\Phi(t,k)}-e^{\Phi(t,\ell)})e^{-\Phi(t,k-\ell)-\Phi(t,\ell)}.$
For $|\ell|\le 2|k-\ell|$, we have
\beno
\langle k\rangle^r|\ell|\leq C\langle k-\ell\rangle^{r+1-\d}|\ell|^{\d},\quad m(k,\ell)\leq C,
\eeno
which imply that 
\begin{align*}
\Big\|\langle k\rangle^r\sum_{|\ell| \le 2|k-\ell|}m(k,\ell)\widehat{f^{+}_\Phi}(k-\ell) |\ell|\widehat{g^{+}_{\Phi}}(\ell)\Big\|_{\ell^2_k}\leq C\|f^+_\Phi\|_{H^{r+1-\d}_x}\|g^{+}_{\Phi}\|_{H^{s+\d}_x}\leq C|f|_{X^{r+1-\delta}}|g|_{X^{s+\delta}}.
\end{align*}
For $|\ell|\geq 2|k-\ell|$, we have
\beno
\langle k\rangle^r\leq C\langle \ell\rangle^r,\quad
m(k,\ell)\leq C\langle \ell\rangle^{\sigma-1}|k-\ell|,
\eeno
which imply that
\begin{align*}
\Big\|\langle k\rangle^r\sum_{|\ell|\ge 2|k-\ell|}m(k,\ell)\widehat{f^{+}_\Phi}(k-\ell) |\ell|\widehat{g^{+}_{\Phi}}(\ell)\Big\|_{\ell^2_k}\leq C\|\pa_xf^+_\Phi\|_{H^{s}_x}\|g^{+}_{\Phi}\|_{H^{r+\sigma}_x}\leq C|f|_{X^{s_1}}|g|_{X^{r+\sigma}}.
\end{align*}
This shows the first inequality.\smallskip

For the second one, we have
\begin{align*}
\|P_{\geq N}(f\pa_xg)_{\Phi}-f(P_{\geq N}\pa_xg_{\Phi})\|_{H^r_x}\leq& \|[P_{\geq N},f\pa_x]g_\Phi\|_{H^r_x}+\|P_{\geq N}[e^{\Phi},f\pa_x]g\|_{H^r_x}\\
=& I_1+I_2.
\end{align*}
It follows from Lemma \ref{lem:com-S} that 
\begin{align*}
I_1\leq C|f|_{X^{s_1}}|P_{\geq \f{N}{2}}g|_{X^r}+C|f|_{X^{r+1-\delta}}|g|_{X^{s+\delta}}.
\end{align*}
Note that $\chi(\f{k}{N})\leq \chi(\f{k-\ell}{N/2})+\chi(\f{\ell}{N/2})$. We infer from the first inequality of this lemma that  
\begin{align*}
I_2\leq C|f|_{X^{s_1}}|P_{\geq \f{N}{2}}g|_{X^{r+\sigma}}+C|f|_{X^{r+1-\delta}}|g|_{X^{s+\delta}}.
\end{align*}
Putting $I_1-I_2$ together, we arrive at the second inequality.
\end{proof}

\subsection{Elliptic equation in a strip}

We denote by $(\tri_{\e,D})^{-1}h$ the solution of the following elliptic equation:
\begin{align*}
\left\{
\begin{aligned}
&\tri_\e F=(\pa_y^2+\e^2\pa_x^2)F=h,\quad (x,y)\in \mathcal{S},\\
 & F|_{y=0,1}=0.
\end{aligned}
\right.
\end{align*}
Let us introduce some notations
\begin{align}
K_1(k,y)=\f{e^{\e |k|y}-e^{-\e |k|y}}{e^{\e |k|}-e^{-\e |k|}},\quad K_2(k,y)=e^{-\e|k|y},\label{def:K12}
\end{align}
and 
\begin{align}
\nonumber
G_0(k,y)=&e^{-\e|k|}K_1(k,y)-K_1(k,1-y)-K_2(k,y)\\
=&\f{2e^{-\e|k|(1-y)}-2e^{\e|k|(1-y)}}{e^{\e |k|}-e^{-\e |k|}},\label{def:G0}\\
\nonumber
G_1(k,y)=&K_1(k,y)-e^{-\e|k|}K_1(k,1-y)+K_2(k,1-y)\\
=&\f{2e^{\e|k|y}-2e^{-\e|k|y}}{e^{\e |k|}-e^{-\e |k|}}.\label{def:G1}
\end{align}

\begin{lemma}\label{lem:elliptic}
Let $F=(\tri_{\e,D})^{-1}h$.
It holds that
\begin{align*}
\widehat{F}(k,y)=&\f{e^{-\e|k|(1-y)}}{2|k|\e}\int_0^1K_1(k,y)\widehat{h}(k,y)dy+\f{e^{-\e|k|y}}{2|k|\e}\int_0^1K_1(k,1-y)\widehat{h}(k,y)dy\\
&+\f{1}{2|k|\e}\int_1^yK_2(k,y'-y)\widehat{h}(k,y')dy'-\f{1}{2|k|\e}\int_0^yK_2(k,y-y')\widehat{h}(k,y')dy'.
\end{align*}
In particular, we have
\begin{align*}
\nonumber
\mathcal{F}(\pa_y(\tri_{\e,D})^{-1}h)=&\f{e^{-\e|k|(1-y)}}{2}\int_0^1K_1(k,y)\widehat{h}(k,y)dy-\f{e^{-\e|k|y}}{2}\int_0^1K_1(k,1-y)\widehat{h}(k,y)dy\\
&+\f{1}{2}\int_1^yK_2(k,y'-y)\widehat{h}(k,y')dy'+\f{1}{2}\int_0^yK_2(k,y-y')\widehat{h}(k,y')dy',
\end{align*}
and 
\begin{align*}
&\mathcal{F}(\pa_y(\tri_{\e,D})^{-1}h)|_{y=0}=\f12\int_0^1(G_0\widehat{h})(k,y)dy,\\
& \mathcal{F}(\pa_y(\tri_{\e,D})^{-1}h)|_{y=1}=\f12\int_0^1(G_1\widehat{h})(k,y)dy.
\end{align*}
\end{lemma}

\begin{proof}
{Taking the Fourier transformation in $x$ on $F$, we get}
\begin{align*}
(\pa_y^2-\e^2|k|^2)\widehat{F}=\widehat{h},\quad \widehat{F}|_{y=0,1}=0.
\end{align*}
Then we have
\begin{align*}
\widehat{F}(k,y)=C_1(k) e^{\e |k|y}+C_2(k) e^{-\e |k| y}+F_*(k,y),
\end{align*}
where
\begin{align*}
F_*(k,y)=\f{1}{2|k|\e}\int_1^y e^{-\e |k|(y'-y)} \widehat{h}(k,y')dy'-\f{1}{2|k|\e}\int_0^y e^{-\e |k|(y-y')} \widehat{h}(k,y')dy',
\end{align*}
and
\begin{align*}
&C_1(k)=\f{1}{2|k|\e(e^{-\e |k|}-e^{\e |k|})}\int_0^1(e^{-\e |k|(1+y)}-e^{-\e |k|(1-y)})\widehat{h}(k,y)dy,\\
&C_2(k)=\f{1}{2|k|\e(e^{-\e |k|}-e^{\e |k|})}\int_0^1(e^{-\e |k|(1-y)}-e^{\e |k|(1-y)})\widehat{h}(k,y)dy.
\end{align*}
Recalling the definitions of $K_1$ and $K_2$, we obtain the solution formula  of $\widehat{F}(k,y)$.
The other formulas can be directly obtained from $\widehat{F}(k,y)$.
\end{proof}

We also introduce
\ben\label{def:G23}
G_2(k,y)=\pa_yG_0(k,y),\quad G_3(k,y)=\pa_y G_1(k,y). 
\een
It is easy to find that for any $s\in[1,\infty],$ 
\begin{align}
&\|(K_1,K_2,G_0,G_1)\|_{L^s}\leq  C\min\Big\{1,\f{1}{\e^{\f 1s} (1+|k|)^{\f 1s}}\Big\},\label{eq:K12-est}\\
&\|(\pa_y K_1,\pa_y K_2, G_2,G_3)\|_{L^s}\leq  C\e^{1-\f 1s} (1+|k|)^{1-\f1s}.\label{eq:G23-est}
\end{align}
Here the constant $C$ is independent of $k, \e$.

\section{The vorticity formulation of the error equations}
We denote
\begin{align*}
u^R=u^\e-u^p,\quad v^R=v^\e-v^p,\quad p^R=p^\e-p^p.
\end{align*}
It is easy to find that
\begin{align}\label{eq:error-u}
\left\{
\begin{aligned}
&\pa_t u^R-\Delta_\e u^R+\pa_x p^R+u^\e\pa_x u^R+u^R\pa_x u^p+v^\e\pa_y u^R+v^R\pa_y u^p-\e^2\pa_x^2 u^p=0,\\
&\e^2(\pa_t v^R-\Delta_\e v^R)+\pa_y p^R+\e^2(\pa_t v^p -\e^2 \pa_x^2 v^p-\pa_y^2 v^p+u^\e\pa_x v^\e+v^\e\pa_y v^\e)=0,\\
&\pa_x u^R+\pa_y v^R=0,\\
&(u^R,v^R)|_{y=0}=(u^R,v^R)|_{y=1}=0,\\
&(u^R,v^R)|_{t=0}=0.
\end{aligned}
\right.
\end{align}

We  introduce the vorticity $\om^R=\pa_y u^R-\e^2 \pa_x v^R$, which satisfies
 \begin{align*}
&\pa_t \om^R-\Delta_\e\om^R+f =N(\om^R, \om^R),
 \end{align*}
 where  $f$ and $N(\om^R, \om^R)$ are defined by
 \begin{align}
&f=f_3-\e^2 (f_1+f_2),\label{eq: f}\\
&N(\om^R, \om^R)=-u^R\pa_x \om^R  -v^R\pa_y \om^R,\label{eq:N1}
\end{align}
with
\begin{align}
&f_1=-(u^R\pa_x^2 v^p+v^R\pa_x\pa_y v^p),\label{eq: f_1}\\
&f_2= -( \pa_t \pa_x v^p-\e^2\pa_x^3 v^p-\pa_y^2\pa_x v^p+\pa_x^2\pa_y u^p+u^p\pa_x^2 v^p+v^p\pa_x\pa_y v^p),\label{eq: f_2}\\
&f_3= u^p\pa_x \om^R+u^R\pa_x \om^p+v^p\pa_y \om^R+v^R\pa_y \om^p,\quad \om^p=\pa_y u^p.\label{eq: f_3}
& \end{align}

Next let us derive the boundary condition of the vorticity.  Thanks to $\pa_x u^R+\pa_yv^R=0$ and $v^R|_{y=0,1}=0$, there exists  $\phi$ so that 
\begin{align*}
-\pa_x\phi=v^R,\quad \pa_y\phi=u^R-\f 1 {2\pi}\int_{\cS} u^Rdxdy,
\end{align*}
Since $\int_{\mathbb{T}} v^Rdx=0,$ the function $\phi$ is periodic in $x$. Thanks to $\pa_x\phi|_{y=0,1}=0$ and $\phi(1,x)-\phi(0,x)=0$, we may assume that $\phi|_{y=0,1}=0$. Thus, there holds that 
\begin{align*}
\tri_\e \phi=\om^R\quad \rm{in}\,\,\cS,\quad \phi|_{y={0,1}}=0.
\end{align*}
This shows that 
\begin{align}
u^R=&\pa_y(\tri_{\e,D})^{-1}\om^R+\f 1 {2\pi}\int_{\cS} u^Rdxdy,\label{eq: Biot u^R }\\
v^R=&-\pa_x(\tri_{\e,D})^{-1}\om^R.\label{eq: Biot v^R }
\end{align}

Motivated by \cite{Mae, WWZ}, we have
\begin{lemma}
It holds that 
\begin{align*}
&(\pa_y+\e |D|)\om^R|_{y=0}=\pa_y(\tri_{\e,D})^{-1}(f-N(\om^R, \om^R))|_{y=0}+\f1 {2\pi}\int_{\cS}\pa_tu^Rdxdy,\\
&(\pa_y-\e |D|)\om^R|_{y=1}=\pa_y(\tri_{\e,D})^{-1}(f-N(\om^R, \om^R))|_{y=1}+\f 1 {2\pi}\int_{\cS}\pa_tu^Rdxdy.
\end{align*}
\end{lemma}

\begin{proof}
We only prove the first equality and the second one is similar. We introduce $\om_{h,0}^R$ which is the harmonic extension of $\om^R|_{y=0}$, i.e., 
\begin{align*}
\left\{
\begin{aligned}
&\tri_\e \om^R_{h,0} =0, \quad x\in \mathbb{T}, \,\,y\in\mathbb{R}_{+},\\
&\om^R_{h,0}|_{y=0}=\om^R|_{y=0}.
\end{aligned}
\right.
\end{align*}
We know that
\begin{align*}
\pa_y\om^R_{h,0}|_{y=0}=-\e|D|\om^R|_{y=0}.
\end{align*}
Taking $\pa_t$ to \eqref{eq: Biot u^R } and using $u^R|_{y=0,1}=0$ and the equation of $\om^R$, we obtain
\begin{align*}%\label{eq: om1}
\nonumber
0=&\pa_tu^R|_{y=0}=\pa_y(\tri_{\e,D})^{-1}\om^R_t|_{y=0}+\f1 {2\pi}\int_{\cS} \pa_tu^Rdy\\
=&\pa_y(\tri_{\e,D})^{-1}\Big(\tri_\e(\om^R-\om^R_{h,0})-f+N(\om^R, \om^R)\Big)\Big|_{y=0}+\f1 {2\pi}\int_{\cS} \pa_tu^Rdy\\
=&\pa_y\om^R|_{y=0}-\pa_y\om^R_{h,0}|_{y=0}-\pa_y(\tri_{\e,D})^{-1} (f-N(\om^R, \om^R))\big|_{y=0}+\f1 {2\pi}\int_{\cS}\pa_tu^Rdy,
\end{align*}
 where we used $(\om^R-\om^R_{h,0})|_{y=0}=0$ and $\Delta_{\e}\om^R_{h,0}=0 $.  This shows the first equality.  
\end{proof}

Based on Lemma \ref{lem:elliptic}, we give more precise formulation of the boundary condition of $\om^R$. Firstly, by the definition of $f_3$  and using the divergence free condition, we obtain
\begin{align*}
f_3=\pa_x(u^p\om^R+\pa_x^{-1}v^R~\pa_y \om^p)+u^R \pa_x\om^p-\pa_x^{-1}v^R~\pa_x\pa_y \om^p+\pa_y(v^p\om^R),
\end{align*}
where $\pa_x^{-1}v^R$ is defined by
\begin{align}
\pa_x^{-1}v^R=
\left\{
\begin{aligned}
&-\int_0^y u^Rdz\quad 0\leq y\leq \f12,\\
&-\int_1^y u^Rdz\quad \f12<y\leq1.
\end{aligned}
\right.
\end{align}
Then by Lemma \ref{lem:elliptic} and integration by parts (note $v^p\om^R|_{y=0,1}=0$), we get 
\begin{align*}
\mathcal{F}(\pa_y(\tri_{\e,D})^{-1}f)|_{y=0}(k)=&\f{ik}{2}\int_0^1 G_0(k,y)\mathcal{F}\Big(u^p\om^R+\pa_x^{-1}v^R~\pa_y \om^p\Big)(k,y)dy\\
&-\f{1}{2}\int_0^1 G_2(k,y)\mathcal{F}\Big(v^p\om^R\Big)(k,y)dy\\
&+\f{1}{2}\int_0^1 G_0(k,y)\mathcal{F}\Big(u^R\pa_x \om^p-\pa_x^{-1}v^R~\pa_x\pa_y \om^p-\e^2(f_1+f_2)\Big)(k,y)dy.
\end{align*}
Thus, we obtain
\begin{align*}
\mathcal{F}\Big((\pa_y+\e|D|) \om^R\Big)|_{y=0}=&ik\mathcal{F}h^0(k)+\mathcal{F}h^0_l(k)-\mathcal{F}\big(\pa_y(\tri_{\e,D})^{-1}(N(\om^R, \om^R))|_{y=0}\big)\\
&+\f1 {2\pi}\int_{\cS}\mathcal{F}(\pa_tu^R)dxdy,
\end{align*}
where
\begin{align}
\mathcal{F}h^0(k)=&\f{1}2\int_0^1 \big(G_0\mathcal{F}(u^p\om^R+\pa_x^{-1}v^R~\pa_y \om^p)\big)(k,y)dy,\label{eq:h0}\\
\nonumber
\mathcal{F}h^0_l(k)=&-\f{1}{2}\int_0^1 \big(G_2\mathcal{F}(v^p\om^R)\big)(k,y)dy\\
&+\f{1}{2}\int_0^1 \big(G_0\mathcal{F}(u^R\pa_x \om^p-\pa_x^{-1}v^R~\pa_x\pa_y \om^p-\e^2(f_1+f_2)\big)(k,y)dy.\label{eq:h2}
\end{align}
Similarly, we have
\begin{align*}
\mathcal{F}\Big((\pa_y -\e|D|)\om^R\Big)|_{y=1}=&ik\mathcal{F}h^1(k)+\mathcal{F}h^1_l(k)-\mathcal{F}\big(\pa_y(\tri_{\e,D})^{-1}(N(\om^R, \om^R))|_{y=1}\big)\\
&+\f1 {2\pi}\int_{\cS}\mathcal{F}(\pa_tu^R)dxdy,
\end{align*}
where
\begin{align}
\mathcal{F}h^1(k)=&\f{1}2 \int_0^1 \big(G_1\mathcal{F}(u^p\om^R+\pa_x^{-1}v^R~\pa_y \om^p)\big)(k,y)dy,\label{eq:h1}\\
\nonumber
\mathcal{F}h^1_l(k)=&-\f{1}{2}\int_0^1 \big(G_3\mathcal{F}(v^p\om^R)\big)(k,y)dy\\
&+\f{1}{2}\int_0^1 \big(G_1\mathcal{F}(u^R\pa_x \om^p-\pa_x^{-1}v^R~\pa_x\pa_y \om^p-\e^2(f_1+f_2)\big)(k,y)dy.\label{eq:h3}
\end{align}

Finally, we conclude that 
 \begin{align}\label{eq:error-om}
\left\{
\begin{aligned}
&\pa_t \om^R-\tri_{\e}\om^R +f=N(\om^R, \om^R),\\
&(\pa_y+\e|D|) \om^R|_{y=0}=\pa_xh^0+h^0_l- \pa_y(\tri_{\e,D})^{-1}(N(\om^R, \om^R))|_{y=0}+{\f1 {2\pi}\int_{\cS}\mathcal{F}(\pa_tu^R)dxdy},\\
&(\pa_y-\e|D|) \om^R|_{y=1}=\pa_xh^1+h^1_l- \pa_y(\tri_{\e,D})^{-1}(N(\om^R, \om^R))|_{y=1}+{\f1 {2\pi}\int_{\cS}\mathcal{F}(\pa_tu^R)dxdy },\\
&\om^R|_{t=0}=0.
\end{aligned}
\right.
\end{align}
Note that $\pa_xh^0$ and $\pa_xh^1$ are the worst terms.\smallskip

Let us compute $\int_{\cS} \pa_tu^R dxdy.$ Using the equation of $u^R$ in \eqref{eq:error-u}, we find that
\begin{align*}
\int_{\cS} \pa_t u^R dxdy=\int_{\cS}\pa_y^2 u^R dxdy=\int_{\cS} \pa_y\om^R dxdy,
\end{align*}
which gives
\begin{align}\label{eq:ut-est}
\Big|\int_{\cS} \pa_tu^R dxdy\Big|\leq& \|\pa_y\om^R\|_{L^1}.
\end{align}

Since $(u^R,v^R)$ satisfies the following elliptic equations
\beno
  \left\{
   \begin{array}{ll}
    \tri_\e u^R=\pa_y\om^R,\\
    u^R|_{y=0,1}=0,
   \end{array}
  \right.
  \qquad
    \left\{
     \begin{array}{ll}
      \tri_\e v^R=-\pa_x \om^R,\\
      v^R|_{y=0,1}=0,
     \end{array}
    \right.
\eeno
we arrive at
\begin{align}\label{eq:uR-elliptic}
\big\|(u^R,\e v^R, \pa_y u^R, \e \pa_x u^R, \e\pa_y v^R,\e^2\pa_x v^R)\big\|_{X^r}\leq C\|\om^R\|_{X^r}.
\end{align}

\section{Boundary layer lift }

To handle the worst terms $\pa_xh^0$ and $\pa_xh^1$, motivated by \cite{GMV}, we introduce a boundary layer lift for the vorticity.
More precisely, we consider the heat equation 
\begin{align}\label{eq:om-b}
\left\{
\begin{aligned}
&(\pa_t-\tri_\e)\om^{b,i}=0,   \\
&\pa_y\om^{b,i}|_{y=i}=\pa_x h^i,\\
&\om^{b,i}|_{t=0}=0,
\end{aligned}
\right.
\end{align}
where $t\in[0,T], ~x\in \mathbb{T}$ and $y>0$ for $i=0$ and $y<1$ for $i=1.$ Here $(h^0,h^1)$ is given by \eqref{eq:h0} and \eqref{eq:h1}. We also introduce the boundary layer velocity $(u^{b,i},v^{b,i})$, which are given by
\begin{align}
&u^{b,0}(x,y)=\int_{+\infty}^y\om^{b,0}(x,z)dz,\quad v^{b,0}=\int_{y}^{+\infty}\pa_x u^{b,0}(x,z)dz \quad \mbox{for}\quad y>0,\label{def:ub-0}\\
&u^{b,1}(x,y)=\int_{-\infty}^y\om^{b,1}(x,z)dz,\quad v^{b,1}=\int_{y}^{-\infty}\pa_x u^{b,1}(x,z)dz \quad \mbox{for}\quad y<1.\label{def:ub-1}
\end{align}

Motivated by Lemma 3.1 in \cite{GMV}, we have the following uniform estimates for $(\om^{b,i}, u^{b,i}, v^{b,i})$ in Gevrey class 
$X^r=X^r_{\sigma, \tau}$. 

\begin{lemma}\label{lem:lift}
Let $T>0$ and $r\in \R$. The boundary layer vorticity $\om^{b,i}$ obeys that
\begin{align*}
&\int_0^t\|\om^{b,i}\|_{ X^r}^2+\|(y-i)\pa_y\om^{b,i}\|_{X^r}^2ds\leq \f{C}{\beta^{\f32}}\int_0^t|h^i|_{X^{r+1-\f{3\sigma}4}}^2ds,\\%\label{est: lift om 1}\\
&\int_0^t\|(y-i)^\ell\om^{b,i}\|_{ X^r}^2+\|(y-i)^{\ell+1}\pa_y\om^{b,i}\|_{ X^r}^2ds\leq \f{C}{\beta^{\f32+\ell}}\int_0^t|h^i|_{ X^{r+1-\f{3\sigma}4-\f {\ell\sigma}2}}^2ds,\\% \label{est: lift om 2}\\
&\int_0^t|\om^{b,i}|_{y=1-i}|_{X^r}^2ds+|\pa_y\om^{b,i}|_{y=1-i}|_{ X^r}^2ds\leq \f{C}{\beta^{2M}}\int_0^t|h^i|_{X^{r+1-M\sigma}}^2ds,\\ 
& \int_0^t\|(\pa_y,\e\pa_x)\om^{b,i}\|_{X^r}^2ds\leq \f{C}{\beta^\f12}\int_0^t| h^i|_{X^{r+1-\f{\s}4}}^2ds,\\
&\sup_{s\in [0,t]}\|\om^{b,i}(s)\|_{X^r}^2\leq \f{C}{\beta^\f12}\int_0^t|h^i|_{X^{r+1-\f{\sigma}4}}^2ds,
\end{align*}
and the boundary layer velocity $u^{b,i}$ obeys that
\begin{align*}
&\int_0^t\|u^{b,i}\|_{ X^r}^2ds\leq \f{C}{\beta^{\f52}}\int_0^t|h^i|_{ X^{r+1-\f{5\sigma}4}}^2ds,\\% \label{est: lift u 1}\\
&\int_0^t\Big|\int_0^\infty u^{b,0}dy\Big|_{ X^r}^2ds+\int_0^t\Big|\int_{-\infty}^1 u^{b,1}\Big|_{ X^r}^2ds \leq \f{C}{\beta^{3}}\int_0^t|h^i|_{ X^{r+1-\f{3\sigma}2}}^2ds,\\% \label{est: lift u 1}\\
&\int_0^t\|(y-i)u^{b,i}\|_{X^r}^2ds\leq \f{C}{\beta^{\f72}}\int_0^t|h^i|_{ X^{r+1-\f{7\sigma}4}}^2ds,\\%\label{est: lift u 2}\\
&\int_0^t\|\e|D|u^{b,i}\|_{ X^r}^2ds\leq \f{C}{\beta^{\f32}}\int_0^t|h^i|_{ X^{r+1-\f{3\sigma}4}}^2ds,\\%\label{est: lift u  3}\\
&\int_0^t|u^{b,i}|_{y=i}|_{ X^r}^2\leq \f{C}{\beta^2}\int_0^t|h^i|_{X^{r+1-\sigma}}^2ds,%\label{est: lift u 4}
\end{align*}
and the boundary layer velocity $v^{b,i}$ obeys that
\begin{align*}
&\int_0^t\|v^{b,i}\|_{X^r}^2ds\leq \f{C}{\beta^{\f72}}\int_0^t|h^i|_{X^{r+2-\f{7\sigma}4}}^2ds,\\%\label{est: lift u 1}\\
&\int_0^t\|\e|D|v^{b,i}\|_{X^r}^2ds\leq \f{C}{\beta^{\f72}}\int_0^t|h^i|_{X^{r+2-\f{5\sigma}4}}^2ds,\\%\label{est: lift u 1}\\
%&\int_0^t\|\e^2|D|v^{b,i}\|_{X^r}^2ds\leq \f{C}{\beta^{\f72}}\int_0^t|h^i|_{X^{r+1-\f{3\sigma}4}}^2ds,\\%\label{est: lift u 1}\\
&\int_0^t|v^{b,i}|_{y=i}|_{X^r}^2ds\leq \f{C}{\beta^{3}}\int_0^t|h^i|_{ X^{r+2-\f{3\sigma}2}}^2ds,\\%\label{est: lift u 1}\\
&\int_0^t|v^{b,i}|_{y=1-i}|_{ X^r}^2ds\leq \f{C}{\beta^{2M}}\int_0^t|h^i|_{X^{r+2-M\sigma}}^2ds,%\label{est: lift u 1}
\end{align*}
for  all $t\in[0,T]$, $i=0,1$ and any $M\geq0.$

\end{lemma}

\begin{proof}
The proof is almost the same as Lemma 3.1 in \cite{GMV}. Here we just show main idea by proving an inequality. 

Thanks to the definition of $\tau(t)$,  we find that $\om^{b,0}_\Phi$ satisfies 
\begin{align}\label{eq:om-b0}
\left\{
\begin{aligned}
&(\pa_t+\tau_0\beta\langle D\rangle^\sigma -\tri_\e) \om^{b,0}_\Phi=0,\quad t\in[0,T], (x,y)\in \mathbb{T}\times [0,\infty), \\
&\pa_y\om^{b,0}_\Phi|_{y=0}=\pa_xh^0_\Phi,\\
&\om^{b,0}_j|_{t=0}=0.
\end{aligned}
\right.
\end{align}
For fixed $x\in\mathbb{T}$, we define that $h^0_\Phi(t,x)=0$ for $t\in \mathbb{R} \backslash [0,T]$, and  then we consider the extended system of \eqref{eq:om-b0}:
\begin{align}\label{eq: extension om^b}
\left\{
\begin{aligned}
&(\pa_t+\tau_0\beta\langle D\rangle^\sigma -\tri_\e) \overline{\om}^{b,0}_\Phi=0,\quad t\in \mathbb{R}, (x,y)\in \mathbb{T}\times(0,\infty),\\
&\pa_y\overline{\om}^{b,0}_\Phi|_{y=0}=\pa_xh^0_\Phi,
\end{aligned}
\right.
\end{align}
which satisfies 
\beno
 \overline{\om}^{b,0}_\Phi= \om^{b,0}_\Phi,\quad t\in[0,T],\quad \textrm{and}\quad  \overline{\om}^{b,0}_\Phi= 0,\quad t<0.
\eeno
which comes from Lemma 3.2 in \cite{GMV}. Taking Fourier transform in $t, x$ to obtain
 \begin{align*}
\left\{
\begin{aligned}
&(i\zeta+\tau_0\beta\langle k\rangle^\sigma +\e^2|k|^2)\mathcal{F}_{t,x}\overline{\om}^{b,0}_\Phi-\pa_y^2\mathcal{F}_{t,x}\overline{\om}^{b,0}_\Phi=0,\\
&\pa_y\mathcal{F}_{t,x}\overline{\om}^{b,0}_\Phi|_{y=0}=ik\mathcal{F}_{t,x}h^{b,0}_\Phi.
\end{aligned}
\right.
\end{align*}
Then the solution is given by
\begin{align}\label{eq: solve om^b,0}
\mathcal{F}_{t,x}\overline{\om}^{b,0}_\Phi(\zeta,k,y)=\f{-ik\mathcal{F}_{t,x}h^{b,0}_\Phi(\zeta, k)}{\sqrt{i\zeta+\tau_0\beta\langle k\rangle^\sigma+\e^2|k|^2}}e^{-y\sqrt{i\zeta+\tau_0\beta\langle k\rangle^\sigma+\e^2|k|^2}}.
\end{align}
For all $\zeta$ with $\Im \zeta\leq0,$ there holds
\begin{align*}
&\Big|\sqrt{\tau_0\beta\langle k\rangle^\sigma+\e^2|k|^2+i\zeta}\Big|\geq \sqrt{\tau_0\beta\k^\sigma+\e^2|k|^2-\Im\zeta}\geq \sqrt{\tau_0\beta\k^\sigma+\e^2|k|^2},
\end{align*}
which along with  \eqref{eq: solve om^b,0} implies 
\beno
\|\mathcal{F}_{t,x}\overline{\om}^{b,0}_\Phi(\zeta,k,y)\|^2_{\ell^2_k(L^2_{y,\zeta})}\leq \big\|\f{C_1|k|}{(\beta\k^\sigma)^{3/4}}\mathcal{F}_{t,x}h^0_{\Phi}\big\|^2_{L^2_\zeta(\ell^2_k)}.
\eeno
This shows by Plancherel's formula that
\beno
\int_0^t\|\om^{b,0}\|_{ X^r}^2ds\leq \f{C_1}{\beta^{\f32}}\int_0^t|h^0|_{X^{r+1-\f{3\sigma}4}}^2ds.
\eeno

The proof of the other inequalities is similar. But the proof of the fifth inequality is similar to Lemma 3.3 in \cite{GMV}.
\end{proof}

\section{Control the boundary layer lift via the interior vorticity}
We introduce the boundary layer profiles
\begin{align}
\om^{bl}(t,x,y)=&\om^{b,0}(t,x,y)+\om^{b,1}(t,x,y),\label{def:om-bl}\\
u^{bl}(t,x,y)=&u^{b,0}(t,x,y)+u^{b,1}(t,x,y),\label{def:u-bl}\\
v^{bl}(t,x,y)=&v^{b,0}(t,x,y)+v^{b,1}(t,x,y)\label{def:v-bl}\\
=&\pa_x\Big(\int_y^{+\infty} u^{b,0}(t,x,z)dz+\int_{y}^{-\infty} u^{b,1}(x,z)dz\Big),\nonumber
\end{align}
and the interior vorticity and velocity as follows
\begin{align}
&\om^{in}(t,x,y)=\om^R(t,x,y)-\om^{bl}(t,x,y),\label{def:om-in}\\
&u^{in}(t,x,y)=u^{R}(t,x,y)-u^{bl}(t,x,y),\label{def:u-in}\\
&v^{in}(t,x,y)=v^R(t,x,y)-v^{bl}(t,x,y).\label{def:v-in}
\end{align}

The following lemma gives the relation between $(u^{in}, v^{in})$ and $\om^{in}$.

\begin{lemma}\label{lem:win-uin}
Let $\Psi$ solve the elliptic equation
\begin{align}\label{eq:Psi}
\left\{
\begin{aligned}
&\tri_\e\Psi=0,\\
& \Psi|_{y=0}=-\Big(\int_0^{+\infty} u^{b,0}(t,x,z)dz+\int_{0}^{-\infty} u^{b,1}(t,x,z)dz\Big),\\
&\Psi_{y=1}=-\Big(\int_1^{+\infty} u^{b,0}(t,x,z)dz+\int_{1}^{-\infty} u^{b,1}(t,x,z)dz\Big).
\end{aligned}
\right.
\end{align}
Then it holds that
\begin{align*}
&u^{in}-\f1 {2\pi}\int_{\cS} u^R dxdy+\pa_y \Psi=\pa_y (\tri_{\e,D})^{-1}(\om^{in}+\e^2\pa_x v^{bl}),\\
& v^{in}-\pa_x \Psi=-\pa_x (\tri_{\e,D})^{-1}(\om^{in}+\e^2\pa_x v^{bl}).
\end{align*}
\end{lemma}

\begin{proof}
By the construction, we have 
\begin{align*}
 \pa_x u^{in}+\pa_y v^{in}=0,\quad  \om^{in}=\pa_y u^{in}-\e^2\pa_x v^{in}-\e^2\pa_x v^{bl}.
  \end{align*}
Thanks to $\pa_x u^{in}+\pa_y v^{in}=0$, there exists a stream function $\phi$ so that 
\begin{align*}
-\pa_x\phi=v^{in},\quad \pa_y \phi=u^{in}-\f1 {2\pi}\int_{\cS} u^R dxdy.
\end{align*}
Here $\phi$ is a periodic function in $x$ due to $\int_{\mathbb{T}}v^{in}dx=0$.
Thanks to $\int_{\T}\phi(x,1) dx=\int_{\T}\phi(x,0)dx$, we may assume that 
$\int_{\T}\phi(x,1) dx=\int_{\T}\phi(x,0)dx=0$.Thus, we find that 
\begin{align*}
\tri_\e (\phi+\Psi)=\om^{in}+\e^2\pa_x v^{bl},\quad (\phi+\Psi)|_{y=0,1}=0.
\end{align*}
 This implies our result.
\end{proof}

\begin{lemma}\label{lem:uin-win}
Let $0<\al_1<\al_2$ and $r\ge 0$. Then it holds that
\begin{align*}
&\int_0^t\|(u^{in},\e v^{in})\|_{X^{r+\f{\sigma}2}}^2ds\leq C\int_0^t\|\om^{in}\|_{X^{r+\f{\sigma}2}}^2ds+\f{C}{\beta^{\f32}}\int_0^t|(h^0,h^1)|_{X^{r+1-\f{3\sigma}4}}^2ds,\\
&\int_0^t\|(P_{\geq \al_2 N(\e)}-P_{\geq \al_1 N(\e)})( u^{R},\e v^{R})\|_{X^r}^2ds\leq C\int_0^t\|(P_{\geq \al_2 N(\e)}-P_{\geq \al_1 N(\e)})\om^{R}\|_{X^{r-\f{\sigma}{2}}}^2ds,
\end{align*}
and  the weighted estimate
\begin{align*}
\int_0^t\|(P_{\geq \al_2 N(\e)}-P_{\geq \al_1 N(\e)})(\varphi u^{R},\e\varphi v^{R})\|_{X^{r}}^2ds&\leq C\int_0^t\Big(\|(P_{\geq \al_2 N(\e)}-P_{\geq \al_1 N(\e)})\om^{R}\|_{X^{r-\sigma}}^2\\
&+\|(P_{\geq \al_2 N(\e)}-P_{\geq \al_1N(\e)})(\varphi\om^{R})\|_{X^{r-\f{\sigma}{2}}}^2\Big)ds,
\end{align*}
where the weight function $\varphi$ is defined by $\varphi(y)=y(1-y)$.
\end{lemma}

\begin{proof}
Using the fact that 
\begin{align*}
\big\|\mathcal{F}\big((\pa_y,\e\pa_x)(\tri_{\e,D})^{-1}f\big)\big\|_{L^2_y}\leq \f{C}{(1+\e|k|)}\|\widehat{f}\|_{L^2_y}
\end{align*}
and Lemma \ref{lem:win-uin}, we infer that
\begin{align*}
\|(\widehat{u^{in}_\Phi},\e\widehat{v^{in}_\Phi})\|_{L^2_y}\leq& \|(\widehat{\pa_y\Psi_\Phi},\e\widehat{\pa_x\Psi_\Phi})\|_{L^2_y}+\f{C}{(1+\e|k|)}\|\widehat{\om^{in}_\Phi}+\e^2\widehat{\pa_x v^{bl}_\Phi}\|_{L^2_y}\\
&+\|u^R\|_{L^2}.
\end{align*}
Thanks to the definition of $\Psi$, we have
 \begin{align}\label{eq: Phi 1}
 \widehat{\Psi}(t,k,y)=\widehat{\Psi}(t,k,y)=K_1(k,y)\widehat{\Psi}|_{y=1}+K_1(k,1-y)\widehat{\Psi}|_{y=0}.
 \end{align}
 By \eqref{eq:K12-est},  \eqref{eq:G23-est} and following the proof of Lemma \ref{lem:lift}, we can deduce that
 \begin{align*}
\|\mathcal{F}_{t,x}(\pa_y\Psi_\Phi,\e\pa_x\Psi_\Phi)\|_{L^2_y}\leq& \|(\pa_y,\e|k|)K_1\|_{L^2_y}|\mathcal{F}_{t,x}\Psi_\Phi|_{y=0,1}|\\
\leq& C(\e|k|)^\f12\Big(\Big|\int_0^{+\infty} \mathcal{F}_{t,x}u^{b,0}_\Phi(\zeta,k,y)dy\Big|+\Big|\int_1^{-\infty} \mathcal{F}_{t,x}u^{b,1}_\Phi(\zeta,k,y)dy\Big|\Big)\\
\leq& C\f{(\e|k|)^\f12|k|( |\mathcal{F}_{t,x}h^0_\Phi|+|\mathcal{F}_{t,x}h^1_\Phi|)}{(\beta\k^\sigma+\e^2|k|^2)^{\f32}}\leq C\f{|k|( |\mathcal{F}_{t,x}h^0_\Phi|+|\mathcal{F}_{t,x}h^1_\Phi|)}{(\beta\k^\sigma+\e^2|k|^2)^{\f54}},
\end{align*}
and 
\begin{align*}
\f{C}{(1+\e|k|)}\|\e^2\mathcal{F}_{t,x}(\pa_x v^{bl}_\Phi)\|_{L^2_y}\leq \f{\e^2 |k|^3( |\mathcal{F}_{t,x}h^0_\Phi|+|\mathcal{F}_{t,x}h^1_\Phi|)}{(1+\e|k|)(\beta\k^\sigma+\e^2|k|^2)^{\f74}}\leq C\f{|k|( |\mathcal{F}_{t,x}h^0_\Phi|+|\mathcal{F}_{t,x}h^1_\Phi|)}{(\beta\k^\sigma+\e^2|k|^2)^{\f54}}.
\end{align*}
This implies that 
\begin{align*}
\int_0^t\|(u^{in},\e v^{in})\|_{X^{r+\f{\sigma}2}}^2ds\leq &C\int_0^t\|\om^{in}\|_{X^{r+\f{\sigma}2}}^2ds+\f{C}{\beta^{\f52}}\int_0^t|(h^0,h^1)|_{X^{r+1-\f{3\sigma}4}}^2ds+\int_0^t\|u^R\|_{L^2}^2ds\\
\leq &C\int_0^t\|\om^{in}\|_{X^{r+\f{\sigma}2}}^2ds+\f{C}{\beta^{\f32}}\int_0^t|(h^0,h^1)|_{X^{r+1-\f{3\sigma}4}}^2ds.
\end{align*}
Here we used the fact that 
\beno
\int_0^t\|u^R\|_{L^2}^2ds\le C\int_0^t\|\om^R\|_{L^2}^2ds\le  \int_0^t\|\om^{in}\|_{X^{r+\f{\sigma}2}}^2ds+\f{C}{\beta^{\f32}}\int_0^t|(h^0,h^1)|_{X^{r+1-\f{3\sigma}4}}^2ds.
\eeno

By \eqref{eq: Biot u^R } and  \eqref{eq: Biot v^R }, we infer that for  $\al_1 N(\e)\leq |k|\leq \al_2 N(\e)$,
\begin{align*}
\|(\widehat{u^{R}_\Phi},\e\widehat{v^{R}_\Phi})(t,k,\cdot)\|_{L^2_y}\leq \f{C}{(1+\e|k|)}\|\widehat{\om^{R}_\Phi}(t,k,\cdot)\|_{L^2_y}
\leq \f{C}{\k^{\f{\s}2}}\|\widehat{\om^{R}_\Phi}(t,k,\cdot)\|_{L^2_y},
\end{align*}
which implies the second inequality.

By \eqref{eq: Biot u^R } and Lemma \ref{lem:elliptic}, we find that for $\al_1 N(\e)\leq |k|\leq \al_2 N(\e)$, 
\begin{align*}
\nonumber
\varphi(y)\widehat{u^R_\Phi}=&\f{\varphi(y)e^{-\e|k|(1-y)}}{2}\int_0^1K_1(k,y)\widehat{\om^R_\Phi}(t,k,y)dy-\f{\varphi(y)e^{-\e|k|y}}{2}\int_0^1K_1(k,1-y)\widehat{\om^R_\Phi}(t,k,y)dy\\
&+\f{\varphi(y)}{2}\int_1^yK_2(k,y'-y)\widehat{\om^R_\Phi}(t,k,y')dy'+\f{\varphi(y)}{2}\int_0^yK_2(k,y-y')\widehat{\om^R_\Phi}(t,k,y')dy'\\
\nonumber
=& B^1+B^2+B^3+B^4.
\end{align*}

By \eqref{eq:K12-est} and $\|\varphi(y)e^{-\e|k|y}\|_{L^2_y}\leq \f{C}{\e^{\f32}(1+|k|)^{\f32}}$, we get
\begin{align*}
\|B^1\|_{L^2_y}+\|B^2\|_{L^2_y}
\leq&  C\|{\varphi(y)e^{-\e|k|y}}\|_{L^2_y}\|K_1\|_{L^2_y}\|\widehat{\om^R_\Phi}\|_{L^2_y}
\leq \f{C}{\e^2(1+|k|)^2}\|\widehat{\om^R_\Phi}\|_{L^2_y}.
\end{align*}
Thanks to $\varphi(y)=\varphi(y')+(y-y')(1-y-y')$,  we get by Young's inequality  that
\begin{align*}
\|B^3\|_{L^2_y}\leq& \Big\|\int_0^y(y-y')K_2(k,y-y')\widehat{\om^{R}_\Phi}(k,y')dy'\Big\|_{L^2_y}+\Big\|\int_0^yK_2(k,y-y')(\varphi\widehat{\om^{R}_\Phi})(k,y')dy'\Big\|_{L^2_y}\\
\leq& \|yK_2(k,y)\|_{L^1_y}\|\widehat{\om^{R}_\Phi}\|_{L^2_y}+\|K_2(k,y)\|_{L^1_y}\|(\varphi\widehat{\om^{R}_\Phi})\|_{L^2_y}\\
\leq &\f{C}{\e^2(1+|k|)^2}\|\widehat{\om^{R}_\Phi}\|_{L^2_y}+\f{C}{\e(1+|k|)}\|(\varphi\widehat{\om^{R}_\Phi})\|_{L^2_y}.
\end{align*}
Similarly, we have
\begin{align*}
\|B^4\|_{L^2_y}
\leq &\f{C}{\e^2(1+|k|)^2}\|\widehat{\om^{R}_\Phi}\|_{L^2_y}+\f{C}{\e(1+|k|)}\|(\varphi\widehat{\om^{R}_\Phi})\|_{L^2_y}.
\end{align*}
Summing up, we infer that for $\al_1 N(\e)\leq |k|\leq \al_2 N(\e)$,
\begin{align*}
\|\varphi \widehat{u^R_\Phi}\|_{L^2_y}\leq \f{C}{\e^2(1+|k|)^2}\|\widehat{\om^{R}_\Phi}\|_{L^2_y}+\f{C}{\e(1+|k|)}\|\varphi\widehat{\om^{R}_\Phi}\|_{L^2_y}.
\end{align*}
Thus, we obtain
\begin{align*}
\|(P_{\geq \al_2 N(\e)}-P_{\geq \al_1 N(\e)})(\varphi u^R)\|_{X^{r}}
\leq C\Big(&\|(P_{\geq \al_2 N(\e)}-P_{\geq \al_1 N(\e)})\varphi \om^R\|_{X^{r-\f{\sigma}{2}}}\\
&+\|(P_{\geq \al_2 N(\e)}-P_{\geq \al_1 N(\e)})\om^R\|_{X^{r-\sigma}}\Big).
\end{align*}
The estimate of $\e\varphi v^R$ is similar.
\end{proof}

\begin{lemma}\label{lem:h01-est}
Let  $u^p$ be given in Theorem \ref{thm:HyNS}  and $r\in [1,N_0-5]$. Then there exists $\beta_*>1$ such that for $\beta\geq \beta_*$ and  $\sigma\in[\f45,1]$, there holds that 
\begin{align}
&\int_0^t |(h^0,h^1)|^2_{X^{r+\f{\sigma}2}}ds
\leq C\int_0^t\|\om^{in}\|_{X^{r+\f{\sigma}2}}^2ds,\label{est: h^01}
\end{align}
 and for $\sigma\in[\f89,1]$,
\begin{align}
&\int_0^t \big|\e|D|(h^0,h^1)\big|^2_{X^{r+1-\f{3\sigma}{4}}}ds\nonumber\\
&\qquad\leq  C\int_0^t\big(\|P_{\geq N(\e)}(\pa_y u^R,\e^2\pa_x v^R) \|_{X^{r+1-\sigma}}^2+\|\om^{in}\|_{X^{r+\f{\sigma}2}}^2\big)ds.\label{est: h^01 L}
\end{align}
\end{lemma}
\begin{proof}

Let us first prove \eqref{est: h^01}. Recalling the definition of $h^0$ in \eqref{eq:h0}, we have
\begin{align*}
|h^0|_{X^{r+\f{\sigma}{2}}}=\|\k^{r+\f{\sigma}{2}}\widehat{h^0_\Phi}\|_{\ell^2_k}\leq& C\|\k^{r+\f{\sigma}{2}}\int_0^1(G_0\mathcal{F}(u^p\om^R)_\Phi)dy\|_{\ell^2_k}\\
&+C\|\k^{r+\f{\sigma}{2}}\int_0^1(G_0\mathcal{F}(\pa_x^{-1}v^R\cdot \pa_y \om^p)_\Phi)dy\|_{\ell^2_k}\\
\leq&C\|u^p\om^R\|_{X^{r+\f{\sigma}{2}}}+C\|\pa_x^{-1}v^R\cdot \pa_y \om^p\|_{X^{r+\f{\sigma}{2}}}.
\end{align*}
By Lemma \ref{lem:product-Gev}, Lemma \ref{lem:lift} and Lemma \ref{lem:uin-win}, we get
\begin{align*}
\int_0^t\|u^p\om^R\|_{X^{r+\f{\sigma}{2}}}^2ds\leq &C\int_0^t\|\varphi \om^R\|_{X^{r+\f{\sigma}{2}}}^2ds\leq C\int_0^t\|\om^{in}\|_{X^{r+\f{\sigma}{2}}}^2ds+C\int_0^t\|\varphi \om^{bl}\|_{X^{r+\f{\sigma}{2}}}^2ds\\
\leq&C\int_0^t\|\om^{in}\|_{X^{r+\f{\sigma}{2}}}^2ds+\f{C}{\beta^\f52}\int_0^t|(h^0,h^1)|_{X^{r+1-\f{3\sigma}{4}}}^2ds,
\end{align*}
and
\begin{align}
\int_0^t\|\pa_x^{-1}v^R\cdot \pa_y \om^p\|_{X^{r+\f{\sigma}{2}}}^2ds\leq &C\int_0^t\|u^R\|_{X^{r+\f{\sigma}{2}}}^2ds\leq C\int_0^t\|u^{in}\|_{X^{r+\f{\sigma}{2}}}^2ds+C\int_0^t\|u^{bl}\|_{X^{r+\f{\sigma}{2}}}^2ds\nonumber\\
\leq&C\int_0^t\|\om^{in}\|_{X^{r+\f{\sigma}{2}}}^2ds+\f{C}{\beta^\f32}\int_0^t|(h^0,h^1)|_{X^{r+1-\f{3\sigma}{4}}}^2ds.\label{eq:uR-est}
\end{align}
This shows that
\begin{align*}
\int_0^t |h^0|^2_{X^{r+\f{\sigma}2}}ds\leq C\int_0^t\|\om^{in}\|_{X^{r+\f{\sigma}{2}}}^2ds+\f{C}{\beta^\f52}\int_0^t|(h^0,h^1)|_{X^{r+1-\f{3\sigma}{4}}}^2ds.
\end{align*}
Similarly, we have
\begin{align*}
\int_0^t |h^1|^2_{X^{r+\f{\sigma}2}}ds\leq C\int_0^t\|\om^{in}\|_{X^{r+\f{\sigma}{2}}}^2ds+\f{C}{\beta^\f52}\int_0^t|(h^0,h^1)|_{X^{r+1-\f{3\sigma}{4}}}^2ds.
\end{align*}
Choosing $\beta\ge \beta^*$ with $\beta^*$ suitably large, we deduce that 
\begin{align}
\int_0^t |(h^0,h^1)|^2_{X^{r+\f{\sigma}2}}ds
\leq C\int_0^t\|\om^{in}\|_{X^{r+\f{\sigma}2}}^2ds,
\end{align}
if $1-\f{3\sigma}{4}\leq \f{\sigma}{2}$ which is equivalent to $\sigma\ge \f 45.$\medskip

Next we prove \eqref{est: h^01 L}. We have
\begin{align*}
\big|\e|D|h^0\big|_{X^{r+1-\f{3\sigma}{4}}}=&\|\k^{r+1-\f{3\sigma}{4}}\e|k| \widehat{h_{\Phi}} \|_{\ell^2_k}\\
\leq&\Big\|(1_{|k|\geq N(\e)}+1_{|k|\leq N(\e)-1})\k^{r+1-\f{3\sigma}{4}}\e|k|\int_0^1G_0\mathcal{F}(u^p\om^R)_\Phi)dy\Big\|_{\ell^2_k}\\
&+\Big\|(1_{|k|\geq N(\e)}+1_{|k|\leq N(\e)-1})\k^{r+1-\f{3\sigma}{4}}\e|k|\int_0^1G_0\mathcal{F}(\pa_x^{-1}v^R\cdot \pa_y \om^p)_\Phi)dy\Big\|_{\ell^2_k}\\
=& I_1+I_2+I_3+I_4.
\end{align*}

For $I_1,$ due to $|k|\geq N(\e)$, we have $\e(1+|k|)\geq c\k^{\f{\sigma}{2}}$ and 
\begin{align}\label{est: var G_0}
\|\e|k|\|G_0\varphi\|_{L^2_y}\leq C\e|k|\min\Big\{1,\f{1}{\e^{\f32}(1+|k|)^{\f32}}\Big\}\leq C\k^{-\f{\sigma}{4}},
\end{align}
which along with Lemma \ref{lem:com-Gev} gives
\begin{align*}
\int_0^t I_1^2ds\leq& \int_0^t \|P_{\geq N(\e)}(\f{u^p}{\varphi}\om^R)\|_{X^{r+1-\sigma}}^2ds\\
\leq &C\int_0^t\|P_{\geq N(\e)}\om^R
\|_{X^{r+1-\sigma}}^2ds+\int_0^t \|P_{\geq N(\e)}(\f{u^p}{\varphi}\om^R)_\Phi-\f{u^p}{\varphi}P_{\geq N(\e)}\om^R_\Phi\|^2_{L^2_y(H^{r+1-\sigma}_x)}ds\\
\leq&C\int_0^t \|P_{\geq N(\e)}(\pa_y u^R,\e^2\pa_x v^R)\|_{X^{r+1-\sigma}}^2ds+\int_0^t \|P_{\geq \f{N(\e)}{2}}\om^R_\Phi\|_{L^2_y(H^{r+1-\sigma+\sigma-1}_x)}^2ds\\
\leq& C\int_0^t \|P_{\geq N(\e)}(\pa_y u^R,\e^2\pa_x v^R)\|_{X^{r+1-\sigma}}^2ds+C\int_0^t \|\om^{R}\|_{X^r}^2ds.
\end{align*}
For $I_3,$ by \eqref{est: var G_0} again and Hardy's inequality
\beno
\|\f{\pa_x^{-1}v^R_\Phi}{\varphi}\|_{L^2_y}\leq C\|\pa_y\pa_x^{-1}v^R_\Phi\|_{L^2_y}\leq C\|u^R_\Phi\|_{L^2_y},
\eeno
we obtain
\begin{align*}
\int_0^t I_3^2ds\leq &C\int_0^t \|P_{\geq N(\e)}(\f{\pa_x^{-1}v^R}{\varphi}\cdot \pa_y \om^p)\|_{X^{r+1-\sigma}}^2ds\\
\leq &C\int_0^t \|P_{\geq N(\e)}u^R\|_{X^{r+1-\s}}^2ds\leq C\int_0^t \|u^R\|_{X^{r+1-\s}}^2ds.
\end{align*}

For $I_2$, using the facts that $\e|k|\|G_0\varphi\|_{L^2_y}\leq C$ and $\e|k|\|G_0\|_{L^2_y}\leq C(\e|k|)^{\f12}\leq C\k^{\f{\s}{4}}$ due to {$|k|\leq 2N(\e),$}  we get by Lemma \ref{lem:lift} and \eqref{est: h^01} that 
\begin{align*}
\int_0^t I_2^2ds\leq& \int_0^t \|{P_{\leq 2N(\e)}}\e|k|(G_0\varphi)(\f{u^p}{\varphi}\om^{in})\|_{X^{r+1-\f{3\sigma}{4}}}^2ds\\
&+\int_0^t \|{P_{\leq 2N(\e)}}\e|k|G_0(\f{u^p}{\varphi}\varphi\om^{bl})\|_{X^{r+1-\f{3\sigma}{4}}}^2ds\\
\leq& \int_0^t \|\om^{in}\|_{X^{r+1-\f{3\sigma}{4}}}^2ds
+\int_0^t \|\varphi\om^{bl}\|_{X^{r+1-\f{\sigma}{2}}}^2ds\leq C\int_0^t \|\om^{in}\|_{X^{r+\f{\sigma}{2}}}^2ds,
\end{align*}
if $1-\f{\sigma}{2}+1-\f{5\s}{4}\leq \f{\s}{2}$ which is equivalent to $\sigma\ge \f89.$
For $I_4$, we have
\begin{align*}
\int_0^t I_4^2ds\leq &C\int_0^t \|(\f{\pa_x^{-1}v^R}{\varphi}\cdot \pa_y \om^p)\|_{X^{r+1-\f{3\sigma}{4}}}^2ds
\leq C\int_0^t \|u^R\|_{X^{r+1-\f{3\sigma}{4}}}^2ds,
\end{align*}

Summing up the estimates of $I_1-I_4$, we conclude that
\begin{align*}
&\int_0^t |\e|D|h^0|_{X^{r+1-\f{3\sigma}{4}}}^2ds\\
&\leq C\int_0^t\|P_{\geq N(\e)}(\pa_y u^R,\e^2\pa_x v^R)\|_{X^{r+1-\sigma}}^2+\|\om^{in}\|_{X^{r+\f{\sigma}2}}^2+ \|\om^{R}\|_{X^r}^2+\|u^R\|_{X^{r+1-\f{3\sigma}{4}}}^2ds.
\end{align*}
On the other hand, by Lemma \ref{lem:lift} and \eqref{est: h^01}, we have
\begin{align*}
\int_0^t \|\om^{R}\|_{X^r}^2ds\leq& \int_0^t \|\om^{in}\|_{X^r}^2ds+\int_0^t \|\om^{bl}\|_{X^r}^2ds\\
\leq&\int_0^t \|\om^{in}\|_{X^r}^2ds+\f{C}{\beta^{\f32}}\int_0^t |(h^0,h^1)|_{X^{r+1-\f{3\sigma}{4}}}^2ds\\
\leq& C\int_0^t \|\om^{in}\|_{X^{r+\f{\sigma}{2}}}^2ds,
\end{align*}
and
\begin{align*}
\int_0^t \|u^R\|_{X^{r+1-\f{3\sigma}{4}}}^2ds\leq& C\int_0^t \|u^{in}\|_{X^{r+1-\f{3\sigma}{4}}}^2ds+C\int_0^t \|u^{bl}\|_{X^{r+1-\f{3\sigma}{4}}}^2ds\\
\leq& C\int_0^t \|\om^{in}\|_{X^{r+\f{\sigma}{2}}}^2ds.
\end{align*}
These show that 
\begin{align*}
&\int_0^t \big|\e|D|h^0\big|_{X^{r+1-\f{3\sigma}{4}}}^2ds\leq C\int_0^t\|P_{\geq N(\e)}(\pa_y u^R,\e^2\pa_x v^R)\|_{X^{r+1-\sigma}}^2+\|\om^{in}\|_{X^{r+\f{\sigma}2}}^2ds.
\end{align*}
The estimate of $h^1$ is similar.
\end{proof}

The following proposition is a direct consequence of Lemma \ref{lem:lift} and Lemma \ref{lem:h01-est}.

\begin{proposition}\label{prop:om-b}
Under the assumptions of Lemma \ref{lem:h01-est}, there holds that 
\begin{align*}
\sup_{s\in[0,t]}&\|\om^{bl}(s)\|_{X^{r-1+\f{3\sigma}4}}^2+\int_0^t\|(\pa_y,\e\pa_x)\om^{bl}\|_{X^{r-1+\f{3\sigma}4}}^2ds+\beta\int_0^t(\|\om^{bl}\|_{X^{r-1+\f{5\sigma}4}}^2+\|\varphi \om^{bl}\|^2_{X^{r-1+\f{7\sigma}4}})ds\\
\leq& \f{C}{\beta^\f12} \int_0^t\|\om^{in}\|_{X^{r+\f{\sigma}2}}^2ds.
\end{align*}
\end{proposition}

Let us conclude this section by the estimates of $h^0_l, h^1_l$.

\begin{lemma}\label{lem:h01l-est}
Under the assumptions of Lemma \ref{lem:h01-est}, there holds that 
\beno
&&\int_0^t |(h^0_l,h^1_l)|_{X^{r}}^2ds\leq Ct \e^4+C\int_0^t\|\om^{in}\|_{X^{r}}^2ds +C \int_0^t\|\pa_y\om^{in}\|_{X^{r}}^2ds.
\eeno
\end{lemma}

\begin{proof}
Recalling the definition of  \eqref{eq:h2}, we have
\begin{align*}
\int_0^t|h^0_l|_{X^r}^2ds\leq&\|\k^r\int_0^1(G_2\mathcal{F}(v^p\om^R)_\Phi)dy\|_{\ell_k^2}^2ds+\|\k^r\int_0^1(G_0\mathcal{F}(u^R\pa_x \om^p)_\Phi)dy\|_{\ell_k^2 }^2ds\\
&+\|\k^r\int_0^1(G_0\mathcal{F}(\pa_x^{-1}v^R\cdot \pa_x\pa_y \om^p)_\Phi)dy\|_{\ell_k^2}^2ds+\e^4\|\k^r\int_0^1(G_0\mathcal{F}(f_1)_\Phi)dy\|_{\ell_k^2}^2ds\\
&+\e^4\|\k^r\int_0^1(G_0\mathcal{F}(f_2)_\Phi)dy\|_{\ell_k^2}^2ds\\
=& I_1+I_2+I_3+I_4+I_5.
\end{align*}

Using the facts that $\|G_2\|_{L^1_y}\leq C$ and
$$\|\f{v^p}{\varphi^2}\|_{L^\infty_y}^2\leq C\|\pa_x\om^p\|_{L^\infty_y}^2\leq C,$$
we get by Lemma \ref{lem:lift} and  Lemma \ref{lem:h01-est} that 
\begin{align*}
I_1\leq&C\int_0^t\|\varphi^2\om^{R}_\Phi\|_{H^r_x(L_y^\infty)}^2ds\leq C\int_0^t\|\om^{in}_\Phi\|_{H^r_x(L_y^\infty)}^2ds+C\int_0^t\|\varphi^2\om^{bl}_\Phi\|_{H^r_x(L_y^\infty)}^2ds\\
\leq& C\int_0^t(\|\om^{in}\|_{X^r}^2+\|\pa_y\om^{in}\|^2_{X^r} )ds+C\int_0^t\big(\|\varphi\om^{bl}\|_{X^r}^2+\|\varphi^2\pa_y\om^{bl}\|^2_{X^r}\big)ds\\
\leq&C\int_0^t(\|\om^{in}\|_{X^r}^2+\|\pa_y\om^{in}\|^2_{X^r} )ds +\f{C}{\beta^{\f52}}\int_0^t|(h^0,h^1)|_{X^{r+1-\f{5\sigma}4}}^2ds\\
\leq &C\int_0^t(\|\om^{in}\|_{X^r}^2+\|\pa_y\om^{in}\|^2_{X^r} )ds.
\end{align*}
Here we use the Gagliardo-Nirenberg inequality
 \begin{align}\label{equality: GN}
\|f\|_{L^\infty_y}\leq C\|f\|_{L^2_y}^\f12\big(\|f\|_{L^2_y}^\f12+\|\pa_yf\|_{L^2_y}^\f12\big).
\end{align}

Similar to  \eqref{eq:uR-est}, we have
\begin{align*}
|I_2|+|I_3|\leq&C\int_0^t\|\om^{in}\|_{X^r}^2ds+\f{C}{\beta^\f52}\int_0^t|(h^0,h^1)|_{X^{r+1-\f{5\sigma}4}}^2ds\leq C\int_0^t\|\om^{in}\|_{X^r}^2ds.
\end{align*}

For $I_4$ and $I_5$, by Lemma \ref{lem:product-Gev} and \eqref{eq:uR-elliptic}, we have 
\begin{align*}
|I_4|\leq&C\e^2\int_0^t\|(\e u^R, \e v^R)\|_{X^r}^2ds\leq C\e^2\int_0^t\|\om^R\|_{X^r}^2ds\\
\leq&C\int_0^t\big(\|\om^{in}\|_{X^r}^2+\f{1}{\beta^{\f52}}|(h^0,h^1)|_{X^{r+1-\f{5\sigma}{4}}}^2\big)ds\leq C\int_0^t\|\om^{in}\|_{X^r}^2ds
\end{align*}
and
\begin{align*}
|I_5|\leq& Ct\e^4.
\end{align*}
Collecting the estimates $I_1-I_5$, we arrive at
\beno
&&\int_0^t |h^0_l|_{X^{r}}^2ds\leq Ct \e^4+C\int_0^t\|\om^{in}\|_{X^{r}}^2ds +C\int_0^t\|\pa_y\om^{in}\|_{X^{r}}^2ds.
\eeno

The estimate of $h^1_l$ is similar.
\end{proof}

\section{Energy estimate via the vorticity equation}

By the construction of $\om^{in}$, we find that 
 \begin{align}\label{eq: om^in}
\left\{
\begin{aligned}
&\pa_t \om^{in}-\Delta_\e\om^{in}+
u^p\pa_x \om^{in}+v^p\pa_y \om^{in}+v^{in}\pa_y \om^p\\
&\qquad+\e^2(f_1+f_2)
=N(\om^R, \om^R)-u^p\pa_x \om^{bl}-u^{R}\pa_x \om^p-v^p\pa_y \om^{bl}-v^{bl}\pa_y \om^p,\\
&\pa_y \om^{in}|_{y=0}=h^0_l-\e|D|\om^{R}|_{y=0}- \pa_y(\tri_{\e,D})^{-1}(N(\om^R, \om^R))|_{y=0}\\
& \qquad \qquad \qquad-(\pa_y+\e|D|) \om^{b,1}|_{y=0}+{\f{1}{2\pi}\int_{\cS} \pa_tu^R dxdy},\\
& \pa_y \om^{in}|_{y=1}=h^1_l+\e|D|\om^{R}|_{y=1}- \pa_y(\tri_{\e,D})^{-1}(N(\om^R, \om^R))|_{y=1}\\
 & \qquad \qquad \qquad-(\pa_y -\e|D|)\om^{b,0}|_{y=1}+{\f{1}{2\pi}\int_{\cS}\pa_tu^R dxdy},\\
&\om^{in}|_{t=0}=0,
\end{aligned}
\right.
\end{align}
where $(h^0_l,h^1_l)$ is given by \eqref{eq:h2} and \eqref{eq:h3}, $(u^p,v^p)$ is the solution of \eqref{eq:HyNS} and $f_1,f_2$ are given by \eqref{eq: f_1}-\eqref{eq: f_2}. For simplicity, we denote $\mathcal{N}=N(\om^R, \om^R)$.

\begin{proposition}\label{prop:om-in}
Let $\sigma\in[\f89,1]$ and $r=N_0-7$. Then there exists $\beta_0>1$ and $\bar\d>0$ so that for all $t\in[0,T]$, $\beta\geq \beta_0$ and  $\d\in(0,\bar\d)$, there holds that 
\begin{align*}
\sup_{s\in[0,t]}&\|\om^{in}(s)\|_{X^r}^2+\int_{0}^t\|(\pa_y,\e\pa_x)\om^{in})\|_{X^r}^2ds+\beta\int_0^t\|\om^{in}\|_{X^{r+\f{\sigma}2}}^2ds\\
\leq&Ct \e^4+2\d\int_0^t\|\mathcal{N}\|_{X^{r-\f{\s}4}}^2ds+C\e^2\int_0^t\|P_{\geq N(\e)}(\pa_y,\e\pa_x)(u^R, \e v^R)\|_{X^{r+1}}^2ds\\
&+C\int_0^t\|P_{\geq N(\e)}(\pa_y,\e\pa_x)(u^R, \e v^R)\|_{X^{r+1-\sigma}}^2ds.
\end{align*}
\end{proposition}

\begin{proof} Firstly, we derive the equation of $\om^{in}_\Phi$:
 \begin{align}\label{eq:om-in-Phi}
\left\{
\begin{aligned}
&\pa_t \om^{in}_\Phi+ \beta \D^\sigma\om^{in}_\Phi-\Delta_\e\om^{in}_\Phi+
(u^p\pa_x+v^p\pa_y) \om^{in}_\Phi+v^{in}_\Phi\pa_y \om^p=-(u^p\pa_x+v^p\pa_y) \om^{bl}_\Phi\\
&\quad-v^{bl}_\Phi\pa_y \om^p
-((u^p\pa_x\om^R)_\Phi-u^p\pa_x\om^R_\Phi)-((v^p\pa_y\om^R)_\Phi-v^p\pa_y\om^R_\Phi)\\
&\quad-((v^R\pa_y\om^p)_\Phi-v^R_\Phi\pa_y\om^p)-(u^R\pa_x \om^p-\e^2u^R\pa_x^2 v^p-\e^2v^R\pa_x\pa_y v^p+\e^2f_2)_\Phi+\mathcal{N}_\Phi
,\\
&\pa_y \om^{in}_\Phi|_{y=0}=(h^0_l)_\Phi-\e|D| \om^{R}_\Phi|_{y=0}- \pa_y(\tri_{\e,D})^{-1}\mathcal{N}_\Phi|_{y=0}\\
&\qquad\qquad\qquad-(\pa_y+\e|D|) \om^{b,1}_\Phi|_{y=0}+\Big({\f{1}{2\pi}}\int_{\cS}\pa_tu^R dxdy\Big)_\Phi,\\
&\pa_y \om^{in}_\Phi|_{y=1}=(h^1_l)_\Phi+\e|D|\om^{R}_\Phi|_{y=1}- \pa_y(\tri_{\e,D})^{-1}\mathcal{N}_\Phi |_{y=1}\\
&\qquad\qquad\qquad-(\pa_y-\e|D|) \om^{b,0}_\Phi|_{y=1}+\Big({\f{1}{2\pi}}\int_{\cS}\pa_tu^R dxdy\Big)_\Phi,\\
&\om^{in}_\Phi|_{t=0}=0.
\end{aligned}
\right.
\end{align}

The worst term in the system is $v_\Phi^{in}\pa_y \om^p$. To handle it, we use the hydrostatic trick. Taking $\D^r$ on the both sides of \eqref{eq:om-in-Phi} and taking $L^2$ inner product  with $\f{\D^r\om^{in}_\Phi}{\pa_y\om^p}$~($\pa_y\om^p\ge \delta_0$), we arrive at
\begin{align*}
\f12\f{d}{dt}&\Big\|\f{\D^r\om^{in}_\Phi}{\sqrt{\pa_y\om^p }}\Big\|_{L^2}^2+\beta \Big\|\f{\D^{r+\f{\sigma}{2}}\om^{in}_\Phi}{\sqrt{\pa_y\om^p}}\Big\|_{L^2}^2+\Big\|\f{(\e\pa_x,\pa_y)\D^r\om^{in}_\Phi}{\sqrt{\pa_y\om^p}}\Big\|_{L^2}^2\\
%%%%%%%%%%%%%%%%%%%%%
=&\int_{\mathcal{S}}\D^r\om_\Phi^{in}\cdot(\e\pa_x,\pa_y)\f{1}{\pa_y\om^p}\cdot (\e\pa_x,\pa_y)\D^r\om_\Phi^{in} dxdy+\int_{\mathbb{T}}\f{\D^r\pa_y\om^{in}_\Phi \D^r\om^{in}_\Phi}{\pa_y\om^p}\Big|_{y=0}^{y=1}dx\\
%%%%%%%%%%%%%%%%%%%%%
&+\int_{\mathcal{S}}|\D^r\om^{in}_\Phi|^2\Big(\pa_x(\f{u^p}{\pa_y\om^p})+\pa_y(\f{v^p}{\pa_y\om^p})\Big)dxdy-\int_{\mathcal{S}}\big[\D^r,u^p\pa_x+v^p\pa_y\big]\om^{in}_\Phi ~\f{\D^r\om^{in}_\Phi}{\pa_y\om^p}dxdy\\
%%%%%%%%%%%%%%%
&-\int_{\mathcal{S}}[\D^r,\pa_y\om^p]v^{in}_\Phi~\f{\D^r\om^{in}_\Phi}{\pa_y\om^p}dxdy
-\int_\mathcal{S} \D^r v^{in}_\Phi\D^r\om^{in}_\Phi dxdy-\int_\mathcal{S}\D^r(u^p\pa_x \om^{bl}_\Phi)~\f{\D^r\om^{in}_\Phi}{\pa_y\om^p} dxdy\\
%%%%%%%%%%%%%%%%%%%
&-\int_\mathcal{S} \D^R( v^p\pa_y \om^{bl}_\Phi)\f{\D^r\om^{in}_\Phi}{\pa_y\om^p} dxdy
-\int_\mathcal{S}\D^r( v^{bl}_\Phi\pa_y\om^{p} )~ \f{\D^r\om^{in}_\Phi}{\pa_y\om^p}dxdy\\
%%%%%%%%%%%%%%%%
&-\int_\mathcal{S}\D^r\Big( (u^p\pa_x\om^R)_\Phi-u^p\pa_x\om^R_\Phi \Big)~ \f{\D^r\om^{in}_\Phi}{\pa_y\om^p}dxdy
-\int_\mathcal{S}\D^r\Big( (v^p\pa_y\om^R)_\Phi-v^p\pa_y\om^R_\Phi \Big)~ \f{\D^r\om^{in}_\Phi}{\pa_y\om^p}dxdy\\
%%%%%
&-\int_\mathcal{S}\D^r\Big( (v^R\pa_y\om^p)_\Phi-v^R_\Phi\pa_y\om^p \Big)~ \f{\D^r\om^{in}_\Phi}{\pa_y\om^p}dxdy\\
%%%%%%%%%
&-\int_\mathcal{S}\D^r\Big( (u^R\pa_x \om^p-\e^2u^R\pa_x^2 v^p-\e^2v^R\pa_x\pa_y v^p)_\Phi\Big)~ \f{\D^r\om^{in}_\Phi}{\pa_y\om^p}dxdy\\
&-\int_\mathcal{S}\e^2\D^r(f_2)_\Phi~ \f{\D^r\om^{in}_\Phi}{\pa_y\om^p}dxdy-\int_\mathcal{S}\D^r\mathcal{N}_\Phi~ \f{\D^r\om^{in}_\Phi}{\pa_y\om^p}dxdy
=T^0+\cdots T^{14}.
\end{align*}
Integrating on $[0,t)$ with $t\leq T$ and using $\pa_y\om^p\geq\d_0$, we obtain
\begin{align*}
\|\om^{in}(t)\|_{X^{r}}^2+2\beta\int_0^t  \|\om^{in}\|_{X^{r+\f{\sigma}2}}^2ds+2\int_0^t\|(\e\pa_x,\pa_y)\om^{in}\|_{X^{r}}^2ds\leq C\int_0^t|T^0|+\cdots+|T^{14}|ds.
\end{align*}

Now we estimate $T^i, i=0,\cdots, 14$.\medskip

\underline{Estimate of $T^0$ and $T^2$}. It is easy to get
\beno
&&\int_0^t|T^0|ds\leq C\int_0^t\|\om^{in}\|_{X^r}^2ds+\delta\int_0^t\|(\e\pa_x,\pa_y)\om^{in})\|_{X^{r}}^2ds,\\
&&\int_0^t|T^2|ds\leq C\int_0^t\|\om^{in}\|_{X^{r}}^2ds.
\eeno

\underline{Estimate of $T^1$}. By Lemma \ref{lem:lift}, Lemma \ref{lem:h01-est}, Lemma \ref{lem:h01l-est} and \eqref{eq:ut-est}, we have
\begin{align*}
&\int_0^t|T^1|ds\\
&\leq C\int_0^t\Big(\Big|\Big(h^0_l,h^1_l,(\pa_y+\e|D| )\om^{b,1}|_{y=0},(\pa_y-\e|D| ) \om^{b,0}|_{y=1}\Big)\Big|_{X^{r-\f{\sigma}4}}\\
&\quad+|\pa_y(\tri_{\e,D})^{-1}\mathcal{N}|_{y=0,1} |_{X^{r-\f{\sigma}4}}+|\int_{\cS} \pa_t u^Rdx dy|_{X^r}\Big)\times\|\om^{in}\|_{X^{r+\f{\sigma}2}}^\f12\|\pa_y\om^{in}\|_{X^{r}}^\f12ds\\
&\quad+\Big|\int_0^t\int_{\mathbb{T}}\e|D|\D^r\om^{R}_\Phi~\f{\D^r\om^{in}_\Phi}{\pa_y\om^p}|_{y=0,1}dxds\Big|\\
&\leq C\int_0^t\big(|(h^0,h^1,h^0_l,h^1_l)|_{X^{r}}+\|\mathcal{N} \|_{X^{r-\f{\s}4}}+\|\pa_y\om^R\|_{L^1}\big)\|\om^{in}\|_{X^{r+\f{\sigma}2}}^\f12\|\pa_y\om^{in}\|_{X^{r}}^\f12ds\\
&\quad+\Big|\int_0^t\int_{\mathbb{T}}\e|D|\D^r\om^{R}_\Phi~\f{\D^r\om^{in}_\Phi}{\pa_y\om^p}|_{y=0,1}dxds\Big|\\
&\leq \delta\int_0^t \|(\pa_y,\e\pa_x)\om^{in}\|_{X^{r}}^2ds+ Ct \e^4+\d\int_0^t\|\mathcal{N}\|_{X^{r-\f{\s}4}}^2ds+C\int_0^t\|\om^{in}\|_{X^{r+\f{\s}2}}^2ds\\
&\quad+\Big|\int_0^t\int_{\mathbb{T}}\e|D|\D^r\om^{R}_\Phi~\f{\D^r\om^{in}_\Phi}{\pa_y\om^p}|_{y=0,1}dxds\Big|,
\end{align*}
where we used
\begin{align*}
\int_0^t\|\pa_y \om^R\|_{L^1}^2ds\leq  C\int_0^t\|\pa_y\om^{in}\|_{L^2}^2ds+\f{C}{\beta^\f12}\int_0^t|( h^0, h^1)|_{X^r}^2ds.
\end{align*}

Let $y_0,y_1\in [0,1]$ so that 
\beno
\big|\e|D|{\om^{in}}(y_0)\big|_{X^r}\le \|\e|D|\om^{in}\|_{X^r},\quad  \big|{\om^{in}}(y_1)\big|_{X^{r+\f \sigma 2}}\le \|{\om^{in}}\|_{X^{r+\f \sigma 2}}.
\eeno
Then we infer that
\begin{align*}
&\Big|\int_0^t\int_{\mathbb{T}}\e|D|\D^r\om^{R}_\Phi~\f{\D^r\om^{in}_\Phi}{\pa_y\om^p}|_{y=0,1}dxds\Big|\\
&\leq \Big|\int_0^t\int_{\mathbb{T}}\e|D|\D^r\om^{in}_\Phi~\f{\D^r\om^{in}_\Phi}{\pa_y\om^p}\Big|_{y=y_0}dxds\Big|+\Big|\int_0^t\int_{\mathbb{T}}\e|D|\D^r\om^{bl}_\Phi~\f{\D^r\om^{in}_\Phi}{\pa_y\om^p}\Big|_{y=y_1}dxds\Big|\\
&\quad+\Big|\int_0^t\int_{0,1}^{y_0}\int_{\T}\pa_y\Big(\e|D|\D^r\om^{in}_\Phi~\f{\D^r\om^{in}_\Phi}{\pa_y\om^p}\Big)dxdyds\Big|+\Big|\int_0^t\int_{0,1}^{y_1}\int_{\T}\pa_y\Big(\e|D|\D^r\om^{bl}_\Phi~\f{\D^r\om^{in}_\Phi}{\pa_y\om^p}\Big)dxdyds\Big|\\
&\leq C\int_0^t\Big(\|\e|D|\om^{in}\|_{X^r}\|\om^{in}\|_{L^\infty_y(H^r_x)}+\|\om^{in}\|_{X^{r+\f{\s}{2}}}\|\e|D|\om^{bl}\|_{L^\infty_y(H^{r-\f{\s}{2}}_x)}\\
&\qquad+\|\e|D|\om^{in}\|_{X^r}\|\pa_y\om^{in}\|_{X^r}+\|\e|D|\om^{bl}\|_{X^r}\|\pa_y\om^{in}\|_{X^r}+\|\e|D|\pa_y\om^{bl}\|_{X^{r-\f{\s}{2}}}\|\om^{in}\|_{X^{r+\f{\s}{2}}}\Big)ds\\
&\leq C\int_0^t\big(\|\e|D|\om^{in}\|_{X^{r}}+\|\e|D| \om^{bl}\|_{X^r}\big)\big(\|\pa_y\om^{in}\|_{X^{r}}+\|\om^{in}\|_{X^{r}}\big)\\
&\qquad+\big(\|\e|D|\om^{bl}\|_{X^{r-\f{\s}2}}+\|\e|D|\pa_y\om^{bl}\|_{X^{r-\f{\s}2}}\big)\|\om^{in}\|_{X^{r+\f{\s}2}}ds.
\end{align*}
By Lemma \ref{lem:lift} and Lemma \ref{lem:h01-est}, we get
\beno
&&\int_0^t\big( \|\e|D|\om^{bl}\|_{X^{r}}^2 +\|\e|D|\pa_y\om^{bl}\|_{X^{r-\f{\s}2}}^2\big)ds\leq C\int_0^t\big|\e|D|(h^0,h^1)\big|^2_{X^{r+1-\f{3\sigma}{4}}}ds\\
&&\leq C\int_0^t\big(\|P_{\geq N(\e)}(\pa_y,\e\pa_x)(u^R, \e v^R)\|_{X^{r+1-\sigma}}^2+\|\om^{in}\|_{X^{r+\f{\sigma}2}}^2\big)ds.
\eeno 

For $\|\e|D|\om^{in}\|^2_{X^{r}}$, we divide the frequency into two parts: $|k|\geq N(\e)$ and $|k|\leq N(\e)$. When $|k|\leq N(\e),$ it holds that $\e|k|\leq C\k^{\f{\sigma}{2}},$ which gives
\begin{align*}
&\int_0^t\|P_{\leq 2N(\e)}\e|D|\om^{in}\|_{X^{r}}^2ds\leq C\int_0^t\|\om^{in}\|_{X^{r+\f{\sigma}{2}}}^2ds.
\end{align*}
For the high frequency part, by Lemma \ref{lem:lift} and Lemma \ref{lem:h01-est}, we get
\begin{align*}
\int_0^t\|P_{\geq N(\e)}\e|D|\om^{in}\|_{X^{r}}^2ds\leq& C\e^2\int_0^t\|P_{\geq N(\e)}(\pa_y,\e\pa_x)(u^R, \e v^R)\|_{X^{r+1}}^2ds+C\int_0^t\| \e|D|\om^{bl}\|_{X^{r}}^2ds\\
&\leq C\e^2\int_0^t\|P_{\geq N(\e)}(\pa_y,\e\pa_x)(u^R, \e v^R)\|_{X^{r+1}}^2ds\\
&\qquad+C\int_0^t|\e|D|(h^0,h^1)|^2_{X^{r+1-\f{3\sigma}{4}}}ds\\
\leq& C\e^2\int_0^t\|P_{\geq N(\e)}(\pa_y,\e\pa_x)(u^R, \e v^R)\|_{X^{r+1}}^2ds\\
&+C\int_0^t\big(\|P_{\geq N(\e)}(\pa_y,\e\pa_x)(u^R, \e v^R)\|_{X^{r+1-\sigma}}^2+\|\om^{in}\|_{X^{r+\f{\sigma}2}}^2\big)ds.
\end{align*}

Summing up, we arrive at 
\begin{align*}
\int_0^t&|T^1|ds\leq \delta\int_0^t \|(\pa_y,\e\pa_x)\om^{in}\|_{X^{r}}^2ds+ Ct \e^4+ C\int_0^t\|\om^{in}\|_{X^{r+\f{\sigma}2}}^2dt'+\d\int_0^t\|\mathcal{N}\|_{X^{r-\f{\s}4}}^2ds\\
&+C\e^2\int_0^t\|P_{\geq N(\e)}(\pa_y,\e\pa_x)(u^R, \e v^R)\|_{X^{r+1}}^2ds+C\int_0^t\|P_{\geq N(\e)}(\pa_y,\e\pa_x)(u^R, \e v^R)\|_{X^{r+1-\sigma}}^2ds.
\end{align*}

\underline{Estimate  of $T^3$}. By Lemma \ref{lem:com-S}, we have
\begin{align*}
\int_0^t|T^3|ds\leq& C\int_0^t(\|\om^{in}\|_{X^{r}}+\|\pa_y\om^{in}\|_{X^{r}})\|\om^{in}\|_{X^{r}}ds\\
\leq& \d\int_0^t \|\pa_y\om^{in}\|_{X^{r}}^2ds+C\int_0^t\|\om^{in}\|_{X^{r}}^2ds.
\end{align*}

\underline{Estimate of $T^4$}. By Lemma \ref{lem:com-S} and Lemma \ref{lem:uin-win}, we have
\begin{align*}
\int_0^t|T^4|ds\leq& C\int_0^t\|v^{in}\|_{X^{r-1}}\|\om^{in}\|_{X^{r}}ds\leq C\int_0^t\|\om^{in}\|_{X^{r}}^2ds
\end{align*}
where we used that fact that $v^{in}=\int_0^y \pa_x u^{in}dz-v^{bl}|_{y=0}$ and 
$$\int_0^t|v^{bl}|_{y=0}|_{X^{r-1}}^2ds\leq C\int_0^t|(h^0,h^1)|_{X^{r+1-\f{3\sigma}{2}}}^2ds\leq C\int_0^t\|\om^{in}\|_{X^{r}}^2ds$$
due to $\sigma\ge \f 89.$\smallskip

\underline{Estimate of $T^5$}. For this term, we need to use the hydrostatic trick. Integration by parts gives
\begin{align*}
T^5=&\int_\mathcal{S} \D^r v^{in}_\Phi~ \D^r(\pa_y u^{in}_\Phi-\e^2\pa_x v^{in}_\Phi-\e^2\pa_x v^{bl}_\Phi)dxdy\\
=&\int_\mathcal{S} \D^r\pa_xu^{in}_\Phi \D^ru^{in}_\Phi dxdy-\int_\mathcal{S} \e^2\D^r\pa_xv^{in}_\Phi~\D^rv^{in}_\Phi dxdy-\e^2\int_\mathcal{S}\D^r v^{in}_\Phi~ \D^r\pa_x v^{bl}_\Phi dxdy\\
&+\int_{\mathbb{T}}\D^rv^{in}_\Phi ~\D^ru^{in}_\Phi\Big|_{y=0}^{y=1}dx\\
=&-\e^2\int_\mathcal{S} \D^r v^{in}_\Phi~\D^r\pa_x v^{bl}_\Phi dxdy+\int_{\mathbb{T}} \D^r v^{in}_\Phi(1)~\D^r u^{in}_\Phi(1)dx-\int_{\mathbb{T}}\D^rv^{in}_\Phi(0)~\D^r u^{in}_\Phi(0)dx\\
=& T^{51}+T^{52}+T^{53}.
\end{align*}
We first consider the boundary term $T^{52}.$ Recalling the boundary condition
 \begin{align*}
 u^{in}(1)=-u^{bl}(1)=-(u^{b,0}(1)+u^{b,1}(1))),~v^{in}(1)=-v^{bl}(1)=-(v^{b,0}(1)+v^{b,1}(1)),
 \end{align*}
it follows from Lemma \ref{lem:lift}  and Lemma \ref{lem:h01-est} that 
\begin{align*}
\int_0^t |T^{52}|ds\leq& C\Big(\int_0^t|v^{b,i}_\Phi|_{y=1}|_{X^{r+1-\f{3\s}2}}^2ds\Big)^{\f12}\Big(\int_0^t|u^{b,i}_\Phi|_{y=1}|_{X^{r-1+\f{3\s}2}}^2ds\Big)^{\f12}\\
\leq&\f{C}{\beta^{\f52}}\Big(\int_0^t|(h^0,h^1)|_{X^{r+3-3\s}}^2ds\Big)^{\f12}\Big(\int_0^t|(h^0,h^1)|_{X^{r+\f{\s}2}}^2ds\Big)^{\f12}
\leq
\f{C}{\beta^{\f52}}\int_0^t\|\om^{in}\|_{X^{r+\f{\s}2}}^2ds,
\end{align*}
here we used $3-3\s\leq\f{\s}{2}$  due to $\sigma\geq \f89$. Similarly, we have
\begin{align*}
\int_0^t |T^{53}|ds\leq
\f{C}{\beta^{\f52}}\int_0^t\|\om^{in}\|_{X^{r+\f{\s}2}}^2ds.
\end{align*}
By Lemma \ref{lem:uin-win} and Lemma \ref{lem:lift}, we get
\begin{align*}
\int_0^t|T^{51}|ds\leq& C\int_0^t\big(\|\e  v^{in}\|_{X^{r+\f{\s}2}}^2+\|\e \pa_x v^{bl}\|_{X^{r-\f{\s}2}}^2\big)ds\\
\leq&C\int_0^t\|\om^{in}\|_{X^{r+\f{\s}2}}^2ds+\f{C}{\beta^{\f32}}\int_0^t|(h^0,h^1)|_{X^{r+1-\f{3\s}4}}^2ds+\f{C}{\beta^{\f52}}\int_0^t|(h^0,h^1)|_{X^{r-\f{\s}2+2-\f{5\s}4}}^2ds\\
\leq& C\int_0^t\|\om^{in}\|_{X^{r+\f{\s}2}}^2ds,
\end{align*}
here we used $-\f{\s}2+2-\f{5\s}4\leq \f{\s}2$ due to $\sigma\ge \f 89$. This shows that 
\begin{align*}
\int_0^t|T^5|ds\leq C\int_0^t\|\om^{in}\|_{X^{r+\f{\s}2}}^2ds.
\end{align*}

\medskip

\underline{Estimates of $T^i, i=6, 7, 8.$} By Lemma \ref{lem:lift}  and  Lemma \ref{lem:h01-est}, we get

\begin{align*}
\int_0^t|T^6|ds\leq& C\int_0^t\|\f{u^p}{\varphi}\pa_x(\varphi\om^{bl})_{\Phi}\|_{H^{r-\f{\sigma}{2},0}}\|\om^{in}_{\Phi}\|_{H^{r+\f{\sigma}{2},0}}ds\\
\leq & C \int_0^t \|\varphi\om^{bl}\|_{X^{r+1-\f{\s}{2}}}^2ds+C\int_0^t\|\om^{in}\|^2_{X^{r+\f{\s}2}} ds\\
\leq& C\int_0^t|(h^0,h^1)|^2_{X^{r+2-\f{7\s}4}}+C\int_0^1\|\om^{in}\|^2_{X^{r+\f{\s}2}} ds\\
\leq& C\int_0^t\|\om^{in}\|^2_{X^{r+\f{\s}2}} ds,
\end{align*}
where we used  $2-\f{7\s}4\leq \f{\s}2$ due to $\sigma\ge \f 89.$ Similarly, we have
\begin{align*}
\int_0^t|T^7|\leq& C\int_0^t\|\f{v^p}{\varphi^2} \varphi^2\pa_y\om^{bl}_\Phi\|_{H^{r-\f{\s}{2},0}}\|\om^{in}_\Phi\|_{H^{r+\f{\s}{2},0}} ds\\
\leq& C\int_0^t\f{1}{\beta^{\f54}}|(h^0,h^1)|_{X^{r+1-\f{7\s}4}}\|\om^{in}\|_{X^{r+\f{\s}2}}ds\\
\leq& C\int_0^t\|\om^{in}\|^2_{X^{r+\f{\s}2}} ds,
\end{align*}
and
\begin{align*}
\int_0^t|T^8|ds\leq& C\int_0^t\|v^{bl}\|_{X^{r+\f{\s}{2}}}\|\om^{in}\|_{X^{r-\f{\s}{2}}}ds\\
\leq& C\int_0^t\f{1}{\beta^{\f74}}|(h^0,h^1)|_{X^{r+2-\f{9\s}4}}\|\om^{in}\|_{X^{r+\f{\s}2}}ds\\
\leq& C\int_0^t\|\om^{in}\|^2_{X^{r+\f{\s}2}} ds.
\end{align*}

\medskip

\underline{Estimates of $T^i, i=9, 10, 11.$} 
By Lemma \ref{lem:com-Gev}, Lemma \ref{lem:lift}  and  Lemma \ref{lem:h01-est}, we have
\begin{align*}
\int_0^t|T^9|ds\leq& C\int_0^t\| (u^p\pa_x\om^R)_\Phi-u^p\pa_x\om^R_\Phi\|_{H^{r-\f{\s}{2},0}}\|\om^{in}\|_{X^{r+\f{\s}{2}}}ds\\
\leq& C\int_0^t\|\f{u^p_\Phi}{\varphi}\|_{L^\infty_y(H^{r+1}_x)}\|\varphi \om^{R}\|_{X^{r+\f{\s}{2}}}\|\om^{in}\|_{X^{r+\f{\s}{2}}}ds\\
\leq& C\int_0^t \|\om^{in}\|_{X^{r+\f{\s}{2}}}^2+\|\varphi \om^{bl}\|_{X^{r+\f{\s}{2}}}^2  ds\leq C\int_0^t \|\om^{in}\|_{X^{r+\f{\s}{2}}}^2ds.
\end{align*}
Due to $v^R=-\int_0^y\pa_x u^Rdz$, we similarly have
\begin{align*}
\int_0^t|T^{11}|ds\leq& C\int_0^t\|(v^R\pa_y\om^p)_\Phi-v^R_\Phi\pa_y\om^p \|_{H^{r-\f{\s}{2},0}}\|\om^{in}\|_{X^{r+\f{\s}{2}}}ds\\
\leq& C\int_0^t\|\pa_y\om^p_\Phi\|_{L^\infty_y(H^{r+1}_x)}\|u^{R}\|_{X^{r+\f{\s}{2}}}\|\om^{in}\|_{X^{r+\f{\s}{2}}}ds\\
\leq& C\int_0^t \|\om^{in}\|_{X^{r+\f{\s}{2}}}^2+\|u^{in}\|_{X^{r+\f{\s}{2}}}^2+\|u^{bl}\|_{X^{r+\f{\s}{2}}}^2  ds\leq C\int_0^t \|\om^{in}\|_{X^{r+\f{\s}{2}}}^2ds,
\end{align*}
and 
\begin{align*}
\int_0^t|T^{10}|ds&\leq C\int_0^t \|\varphi^2\pa_y\om^R\|_{X^r}\|\om^{in}\|_{X^r}ds\leq \d\int_0^t \|\pa_y\om^{in}\|_{X^{r}}^2ds+C\int_0^t \|\om^{in}\|_{X^{r+\f{\s}{2}}}^2ds.
\end{align*}

\underline{Estimates of $T^{i}, i=12,13,14$}.
By Lemma \ref{lem:product-Gev}, Lemma \ref{lem:lift} and Lemma \ref{lem:uin-win}, it is easy to see that
\beno
\int_0^t|T_j^{12}| ds\leq C\int_0^t \|\om^{in}\|_{X^{r+\f{\s}2}}^2 ds+C\int_0^t \|(u^R,\e v^R)\|^2_{X^{r-\f{\s}2}}ds
\leq C\int_0^t \|\om^{in}\|_{X^{r+\f{\s}{2}}}^2ds,
\eeno
and 
\beno
&&\int_0^t|T^{13}|ds \leq  Ct\e^4+C\int_0^t \|\om^{in}\|_{X^{r+\f{\s}2}}^2 ds,\\
&&\int_0^t|T^{14}|\leq C\int_0^t\|\om^{in}\|_{X^{r+\f{\s}2}}^2ds+\d\int_0^t\|\mathcal{N}\|_{X^{r-\f{\s}2}}^2ds.
\eeno

Summing up the estimates of $T^0-T^{14}$, and taking $\beta$ large enough and $\d$ small enough, we deduce the desired result.
\end{proof}

We directly deduce from Proposition \ref{prop:om-b} and Proposition \ref{prop:om-in}  that

\begin{corol}\label{cor:omR}
Under the assumption of Proposition \ref{prop:om-in}, there holds that
\begin{align*}
\sup_{s\in[0,t]}&\|\om^{R}(s)\|_{X^{r-1+\f{3\s}4}}^2++\int_0^t\|(\pa_y,\e\pa_x)\om^{R}\|_{X^{r-1+\f{3\sigma}4}}^2ds+\beta\int_0^t\big(\|\om^{R}\|_{X^{r-1+\f{5\s}4}}^2+\|\varphi \om^{R}\|_{X^{r+\f{\s}{2}}}^2\big)ds\\
 \leq &Ct \e^4+2\d\int_0^t\|\mathcal{N}\|_{X^{r-\f{\s}4}}^2ds+ C\e^2\int_0^t\|P_{\geq N(\e)}(\pa_y,\e\pa_x)(u^R, \e v^R)\|_{X^{r+1}}^2ds\\
&+C\int_0^t\|P_{\geq N(\e)}(\pa_y,\e\pa_x)(u^R, \e v^R)\|_{X^{r+1-\sigma}}^2 ds.
\end{align*}
\end{corol}

\section{Energy estimate via the velocity equation}

In this section, we are devoted to the estimates for the high frequency part of $(u^R,\e v^R)$. In this case, we can directly use the velocity equation. Recall that $(u^R, v^R)$ satisfies
\begin{align}\label{eq: error-(u,v)-1}
\left\{
\begin{aligned}
&\pa_t u^R-{\Delta_\e u^R}+\pa_x p^R+u^p\pa_x u^R+u^R\pa_x u^p+v^R\pa_y u^p+v^p\pa_y u^R\\
&\qquad\qquad+\mathcal{N}_u-\e^2\pa_x^2 u^p=0,\\
&\e^2(\pa_t v^R-\Delta_\e v^R+u^p\pa_x v^R+u^R\pa_x v^p+v^R\pa_y v^p+v^p\pa_y v^R+\mathcal{N}_v)\\
&\qquad\qquad+\pa_y p^R+\e^2(\pa_t v^p -\e^2 \pa_x^2 v^p-\pa_y^2 v^p+u^p\pa_x v^p+v^p\pa_y v^p)=0,\\
&\pa_x u^R+\pa_y v^R=0 ,\\
&(u^R,v^R)|_{y=0}=(u^R,v^R)|_{y=1}=0,\\
&(u^R,v^R)|_{t=0}=0.
\end{aligned}
\right.
\end{align}
Here $(\mathcal{N}_u,\mathcal{N}_v)$ is nonlinear term  given  by
\begin{align*}
\mathcal{N}_u=&u^R\pa_x u^R+v^R\pa_y u^R,\quad \mathcal{N}_v=u^R\pa_x v^R+v^R\pa_y v^R.
\end{align*}

\begin{proposition}\label{prop:ur-high}
Let $\sigma\in[\f45,1]$ and $r=N_0-7$. Then there exist $\beta_1$ and  $T,\bar\d>0$, such that for any  $\d\in(0,\bar\d),~\beta\geq \beta_1$ and $t\in[0,T]$, there holds that
\begin{align*}
&\e^2\|P_{\geq N(\e)}(u^R,\e v^R)(t)\|_{X^{r+1}}^2+\beta \e^2 \int_0^t\|P_{\geq N(\e)}(u^R,\e v^R)\|_{X^{r+1+\f{\s}{2}}}^2\\
&\qquad +\int_0^t\e^2\|P_{\geq N(\e)}(\pa_y,\e\pa_x)( u^{R},\e v^R)\|_{X^{r+1}}^2\\
&\leq  C\int_0^t \|\om^{in}\|_{X^{r+\f{\s}2}}^2ds+\d\int_0^t\|P_{\geq N(\e)}(\mathcal{N}_u,\e \mathcal{N}_v)\|_{X^{r+1-\f{\s}2}}^2ds,
\end{align*}
and
\begin{align*}
&\sup_{s\in[0,t]}\|P_{\geq N(\e)}(u^R,\e v^R)(t)\|_{X^{r+1-\sigma}}^2+\beta  \int_0^t\|P_{\geq N(\e)}(u^R,\e v^R)\|_{X^{r+1-\f{\s}{2}}}^2\\
\nonumber
&\qquad +\int_0^t\|P_{\geq N(\e)}(\pa_y,\e\pa_x)( u^{R},\e v^R)\|_{X^{r+1-\s}}^2\\
\nonumber
&\leq  C\int_0^t \|\om^{in}\|_{X^{r+\f{\s}2}}^2 ds+\d\int_0^t\|P_{\geq N(\e)}(\mathcal{N}_u,\e \mathcal{N}_v)\|_{X^{r+1-\f{3\s}2}}^2ds.
\end{align*}
\end{proposition}

\begin{proof}
Acting operator $e^{\Phi(t,D)}P_{\geq N(\e)}$ on the first two equations of \eqref{eq: error-(u,v)-1} and
taking $H^{r+1,0}$ inner product  with $P_{\geq N(\e)}(u^{R}_\Phi,v^R_\Phi)$, we get by integration by parts that
\begin{align*}
&\f12\f{d}{dt}\|P_{\geq N(\e)}(u^R,\e v^R)(t)\|_{X^{r+1}}^2+\beta  \|P_{\geq N(\e)}(u^R,\e v^R)\|_{X^{r+1+\f{\s}{2}}}^2+\|P_{\geq N(\e)}(\pa_y,\e\pa_x)( u^{R},\e v^R)\|_{X^{r+1}}^2\\
 &=\Big\langle P_{\geq N(\e)} p^R_\Phi, P_{\geq N(\e)}(\pa_x u^R_\Phi+\pa_y v^R_\Phi)\Big\rangle_{H^{r+1,0}}+\f12\Big\langle |P_{\geq N(\e)}u^{R}_\Phi|^2+|\e P_{\geq N(\e)}v^{R}_\Phi|^2,\pa_xu^p+\pa_yv^p\Big\rangle_{H^{r+1,0}}\\
 &\quad-\Big\langle P_{\geq N(\e)}(v^R\pa_yu^p)_\Phi, P_{\geq N(\e)}u^R_\Phi\Big\rangle_{H^{r+1,0}}-\Big\langle P_{\geq N(\e)}\e^2(v^R\pa_yv^p)_\Phi, P_{\geq N(\e)}v^R_\Phi\Big\rangle_{H^{r+1,0}}\\
 &\quad-\Big\langle P_{\geq N(\e)}(u^R\pa_x  u^p)_\Phi, P_{\geq N(\e)}u^R_\Phi\Big\rangle_{H^{r+1,0}}-\Big\langle P_{\geq N(\e)}\e^2(u^R\pa_x  v^p)_\Phi, P_{\geq N(\e)}v^R_\Phi\Big\rangle_{H^{r+1,0}}\\
 &\quad-\Big\langle P_{\geq N(\e)}(u^p\pa_x  u^R)_\Phi-u^p\pa_x  P_{\geq N(\e)}u^R_\Phi, P_{\geq N(\e)}u^R_\Phi\Big\rangle_{H^{r+1,0}}\\
 &\quad-\e^2\Big\langle P_{\geq N(\e)}(u^p\pa_x  v^R)_\Phi-u^p\pa_x P_{\geq N(\e)} v^R_\Phi, P_{\geq N(\e)}v^R_\Phi\Big\rangle_{H^{r+1,0}}\\
 &\quad-\Big\langle P_{\geq N(\e)}(v^p\pa_y  u^R)_\Phi-v^p P_{\geq N(\e)}\pa_y  u^R_\Phi\Big), P_{\geq N(\e)}u^R_\Phi\Big\rangle_{H^{r+1,0}}\\
 &\quad-\e^2\Big\langle P_{\geq N(\e)}(v^p\pa_y  v^R)_\Phi-v^p P_{\geq N(\e)} \pa_y v^R_\Phi, P_{\geq N(\e)}v^R_\Phi\Big\rangle_{H^{r+1,0}}\\
 &\quad+\e^2\Big\langle P_{\geq N(\e)}\pa_x^2  u^p_\Phi, P_{\geq N(\e)}u^R_\Phi\Big\rangle_{H^{r+1,0}}\\
 &\quad-\e^2\Big\langle P_{\geq N(\e)}(\pa_t v^p -\e^2 \pa_x^2 v^p-\pa_y^2 v^p+u^p\pa_x v^p+v^p\pa_y v^p)_\Phi, P_{\geq N(\e)}v^R_\Phi\Big\rangle_{H^{r+1,0}}\\
 &\quad-\Big\langle P_{\geq N(\e)}(\mathcal{N}_u)_\Phi, P_{\geq N(\e)}u^R_\Phi\Big\rangle_{H^{r+1,0}}-\e^2\Big\langle P_{\geq N(\e)}(\mathcal{N}_v)_\Phi, P_{\geq N(\e)}v^R_\Phi\Big\rangle_{H^{r+1,0}}\\
 &=S^1+\cdots+S^{14}.
\end{align*}
This gives
\begin{align*}
&\|P_{\geq N(\e)}(u^R,\e v^R)(t)\|_{X^{r+1}}^2+2\beta  \int_0^t\|P_{\geq N(\e)}(u^R,\e v^R)\|_{X^{r+1+\f{\s}{2}}}^2ds\\
&\quad+2\int_0^t\|P_{\geq N(\e)}(\pa_y,\e\pa_x)( u^{R},\e v^R)\|_{X^{r+1}}^2ds\leq C\int_0^t|S^1|+\cdots+|S^{14}|ds.
\end{align*}

Thanks to $\pa_x u^p+\pa_y v^p=0,~\pa_x u^R+\pa_y v^R=0, $ we have $S^1=S^2=0.$\smallskip

{\underline{Estimate of $S^3-S^6.$}} We get by Lemma \ref{lem:com-Gev}, Lemma \ref{lem:uin-win} and \eqref{eq:uR-elliptic} that 

\begin{align*}
\int_0^t|S^3|ds\leq& \int_0^t\Big(\|P_{\geq N(\e)}v^R_\Phi\pa_yu^p\|_{H^{r+1-\f{\s}{2},0}}+\|P_{\geq N(\e)}(v^R \pa_yu^p)_\Phi-P_{\geq N(\e)}v^R_\Phi\pa_yu^p\|_{H^{r+1-\f{\s}{2},0}}\Big)\\
&\qquad\times\|P_{\geq N(\e)} u^R\|_{X^{r+1+\f{\s}{2}}}ds\\
\leq&\int_0^t\|P_{\geq N(\e)} u^R\|_{X^{r+1+\f{\s}{2}}}^2+\|P_{\geq N(\e)} v^R\|_{X^{r+1-\f{\s}{2}}}^2+\|P_{\geq \f{N(\e)}{2}} v^R\|_{X^{r+\f{\s}{2}}}^2+\|v^R\|_{X^{\f12+}}^2ds\\
\leq&\int_0^t\|P_{\geq N(\e)} u^R\|_{X^{r+1+\f{\s}{2}}}^2+\|P_{\geq N(\e)} \e v^R\|_{X^{r+2-\s}}^2+\|P_{\geq \f{N(\e)}{2}} \e v^R\|_{X^{r+1}}^2+\|\om^R\|_{X^{\f32+}}^2ds\\
  \leq & C\int_0^t\Big(  \|P_{\geq N(\e)}(u^R,\e v^R)\|_{X^{r+1+\f{\sigma}2}}^2 + \|P_{\leq N(\e)}\om^R\|_{X^{r+1-\f{\s}2}}^2+\|\om^R\|_{X^{\f32+}}^2\Big)ds.
\end{align*}
Here we used $2-\s\leq 1+\f{\s}{2}$ and $\e|k|\geq \k^{\f{\s}{2}}$ for $|k|\geq N(\e)$. Similarly, we have
 \begin{align*}
\int_0^t|S^4|ds  \leq & C\int_0^t\Big(  \|P_{\geq N(\e)}(\e v^R)\|_{X^{r+1+\f{\sigma}2}}^2 + \|P_{\leq N(\e)}\om^R\|_{X^{r+1-\f{\s}2}}^2+\|\om^R\|_{X^{\f12+}}^2\Big)ds,
\end{align*}
and 
 \begin{align*}
\int_0^t|S^5|+|S^6|ds  \leq & C\int_0^t\Big(  \|P_{\geq N(\e)} (u^R,\e v^R)\|_{X^{r+1+\f{\sigma}2}}^2 + \|P_{\geq\f{N(\e)}{2}}u^R\|_{X^{r+\f{\s}2}}^2+\|u^R\|_{X^{\f12+}}^2\Big)ds\\
\leq&C\int_0^t\Big(  \|P_{\geq N(\e)} (u^R,\e v^R)\|_{X^{r+1+\f{\sigma}2}}^2 + \|P_{\leq N(\e)}\om^R\|_{X^{r}}^2+\|\om^R\|_{X^{\f12+}}^2\Big)ds.
\end{align*}

{\underline{Estimate of $S^7-S^{10}.$}}
By  Lemma \ref{lem:com-Gev}, Lemma \ref{lem:uin-win} and \eqref{eq:uR-elliptic}, we have
\begin{align*}
\int_0^t|S^7| ds \leq&\int_0^t\|P_{\geq N(\e)}(u^p\pa_x  u^R)_\Phi-u^p\pa_x  P_{\geq N(\e)}u^R_\Phi\|_{H^{r+1-\f{\sigma}{2},0}}\|P_{\geq N(\e)} u^R\|_{X^{r+1+\f{\sigma}{2}}}ds\\
\leq&  C\int_0^t \Big(\|P_{\geq N(\e)} u^R\|_{X^{r+1+\f{\sigma}{2}}}^2+\|P_{\geq \f{N(\e)}{2}} \varphi u^R\|_{X^{r+1+\f{\sigma}{2}}}^2+\| u^R\|_{X^{\f32+}}^2\Big)ds\\
\leq&  C\int_0^t \Big(\|P_{\geq N(\e)} u^R\|_{X^{r+1+\f{\sigma}{2}}}^2+ \|P_{\leq N(\e)}(\varphi\om^R)\|_{X^{r+1}}^2\\
&\qquad\quad+\|P_{\leq N(\e)}\om^R\|_{X^{r+1-\f{\s}2}}^2 +\|\om^R\|^2_{X^{\f32+}}\Big)ds.
\end{align*}
Similarly, we have
\begin{align*}
\int_0^t|S^8| ds
\leq C\int_0^t \Big(&\|P_{\geq N(\e)}(\e v^R)\|_{X^{r+1+\f{\sigma}{2}}}^2+ \|P_{\leq N(\e)}(\varphi\om^R)\|_{X^{r+1}}^2\\
&\quad+\|P_{\leq N(\e)}\om^R\|_{X^{r+1-\f{\s}2}}^2 + \|\om^R\|^2_{X^{\f32+}} \Big)ds,
\end{align*}
and 
\begin{align*}
\int_0^t|S^9|ds\leq& C\int_0^t\|P_{\geq N(\e)} u^R\|_{X^{r+1+\f{\s}2}}^2ds\\
&\quad+\d\int_0^t\Big(\|P_{\ge N(\e)}(\pa_y u^R)\|_{X^{r+1}}^2+\| P_{\leq N(\e)}\pa_y u^R\|_{X^{r+\f{\s}2}}^2+\| \pa_yu^R\|_{X^{\f12+}}^2\Big)ds\\
\leq& C\int_0^t \Big(\|P_{\geq N(\e)}u^R\|_{X^{r+1+\f{\sigma}{2}}}^2+\|P_{\leq N(\e)}\om^R\|_{X^{r+\f{\s}2}}^2  +\|\om^R\|^2_{X^{\f12+}} \Big)ds\\
&\quad+\d\int_0^t\|  P_{\ge N(\e)}(\pa_y u^R)\|_{X^{r+1}}^2ds,
\end{align*}
and 
\begin{align*}
\int_0^t|S^{10}|ds
\leq& C\int_0^t \Big(\|P_{\geq N(\e)}(\e v^R)\|_{X^{r+1+\f{\sigma}{2}}}^2+\|P_{\leq N(\e)}\om^R\|_{X^{r+\f{\s}2}}^2  +\|\om^R\|^2_{X^{\f12+}} \Big)ds\\
&\quad+\d\int_0^t\| P_{\ge N(\e)}(\e \pa_y v^R)\|_{X^{r+1}}^2ds.
\end{align*}

{\underline{Estimate of $S^{11}-S^{12}$.}} It is easy to see that 
\begin{align*}
\int_0^t|S^{11}|+|S^{12}|ds\leq&Ct\e^4+C\int_0^t\| P_{\geq N(\e)}u^R\|_{X^{r+1+\f{\s}2}}^2ds,
\end{align*}
where we used $v^R=-\int_0^y\pa_xu^Rdy'$ and integration by parts for $S^{12}$.\smallskip

{\underline{Estimate of $S^{13}-S^{14}$.}} It is easy to see that 
\begin{align*}
\int_0^t|S^{13}|+|S^{14}|ds\leq C\int_0^t\|P_{\geq N(\e)} (u^R,\e v^R)\|_{X^{r+1+\f{\sigma}2}}^2ds+\d\int_0^t\|P_{\geq N(\e)}(\mathcal{N}_u,\e \mathcal{N}_v)\|_{X^{r+1-\f{\s}2}}ds.
\end{align*}

Summing up the estimates of $S^1-S^{14}$, and then taking $\beta$ large enough and $\d$ small enough,  we arrive at
\begin{align}\label{est: (u^R, v^R) 1}
&\|P_{\geq N(\e)}(u^R,\e v^R)(t)\|_{X^{r+1}}^2+\beta  \int_0^t\|P_{\geq N(\e)}(u^R,\e v^R)\|_{X^{r+1+\f{\s}{2}}}^2\\
&\quad+\int_0^t\|P_{\geq N(\e)}(\pa_y,\e\pa_x)( u^{R},\e v^R)\|_{X^{r+1}}^2\nonumber\\
\nonumber
&\leq  C\int_0^t\Big(\|P_{\leq N(\e)}(\varphi\om^R)\|_{X^{r+1}}^2+\|P_{\leq N(\e)}\om^R\|_{X^{r+1-\f{\s}2}}^2+\|\om^R\|^2_{X^{\f32+}}\Big)ds\\
\nonumber
&\qquad+\d\int_0^t\|P_{\geq N(\e)}(\mathcal{N}_u,\e \mathcal{N}_v)\|_{X^{r+1-\f{\s}2}}^2ds.
\end{align}
By Lemma \ref{lem:lift} and Lemma \ref{lem:h01-est}, we have
\begin{align*}
&\int_0^t\|\varphi\om^R\|_{X^{r+\f{\s}2}}^2ds\leq C\int_0^t\|\om^{in}\|_{X^{r+\f{\s}2}}^2+\|\varphi\om^{bl}\|_{X^{r+\f{\s}2}}^2ds\leq C\int_0^t\|\om^{in}\|_{X^{r+\f{\s}2}}^2ds,\\
&\int_0^t\|\om^R\|_{X^{\f32+}}^2ds\leq \int_0^t\|\om^R\|_{X^{r}}^2ds\leq \int_0^t\|\om^{in}\|_{X^{r}}^2+\|\om^{bl}\|_{X^{r}}^2ds\leq C\int_0^t\|\om^{in}\|_{X^{r+\f{\s}2}}^2ds,
\end{align*}
which along with   $\e|k|\leq \k^{\f{\s}2}$ for  $|k|\leq N(\e)$ give
\begin{align*}
 \int_0^t\e^2\|P_{\leq N(\e)}(\varphi\om^R)\|_{X^{r+1}}^2+\e^2\|P_{\leq N(\e)}\om^R\|_{X^{r+1-\f{\s}2}}^2ds\leq &C\int_0^t
 \|\varphi\om^R\|_{X^{r+\f{\s}2}}^2+\|\om^R\|_{X^{r}}^2ds\\
 \leq& C\int_0^t\|\om^{in}\|_{X^{r+\f{\s}2}}^2ds.
 \end{align*}
Therefore, it holds that 
\begin{align*}
&\e^2\|P_{\geq N(\e)}(u^R,\e v^R)(t)\|_{X^{r+1}}^2+\beta \e^2 \int_0^t\|P_{\geq N(\e)}(u^R,\e v^R)\|_{X^{r+1+\f{\s}{2}}}^2\\
&\qquad +\int_0^t\e^2\|P_{\geq N(\e)}(\pa_y,\e\pa_x)( u^{R},\e v^R)\|_{X^{r+1}}^2\\
&\leq  C\int_0^t \|\om^{in}\|_{X^{r+\f{\s}2}}^2ds +\d\int_0^t\|P_{\geq N(\e)}(\mathcal{N}_u,\e \mathcal{N}_v)\|_{X^{r+1-\f{\s}2}}^2ds,
\end{align*}
which gives the first result.

The second result follows by taking $r-\sigma$ instead of $r$ in \eqref{est: (u^R, v^R) 1}  and noticing that 
 \begin{align*}
 \int_0^t\|P_{\leq N(\e)}(\varphi\om^R)\|_{X^{r+1-\s}}^2+\|P_{\leq N(\e)}\om^R\|_{X^{r+1-\f{3\s}2}}^2ds\leq &C\int_0^t
 \|\varphi\om^R\|_{X^{r+\f{\s}2}}^2+\|\om^R\|_{X^{r}}^2ds\\
 \leq& C\int_0^t\|\om^{in}\|_{X^{r+\f{\s}2}}^2ds.
 \end{align*} 

This completes the proof of the proposition.
\end{proof}

\section{Nonlinear estimates}

In this section, we estimate nonlinear terms $(\mathcal{N},\mathcal{N}_u,\mathcal{N}_v)$, which are defined by
\begin{align*}
&\mathcal{N}=-u^R\pa_x \om^R  -v^R\pa_y \om^R,\\
&\mathcal{N}_u=u^R\pa_x u^R+v^R\pa_y u^R,\quad \mathcal{N}_v=u^R\pa_x v^R+v^R\pa_y v^R.
\end{align*}
For this, let us first assume the following energy bounds:
\begin{align}\label{ass:bootstrap}
\sup_{s\in[0,t]}\|\om^R(s)\|_{X^{r-1}}^2+\int_0^t\|(\pa_y \om^R,\e\pa_x\om^R)\|_{X^{r-1}}^2\leq \mathfrak{C}\e^4
\end{align}
for any $t\in[0,T].$

\begin{proposition}\label{prop:non}
Under the assumption \eqref{ass:bootstrap},  there holds that 
\begin{align*}
&\int_0^t\|\mathcal{N}\|_{X^{r-\f{\s}4}}^2ds\leq C\sup_{s\in[0,t]}\big(\|P_{\geq N(\e)}\e u^R\|_{X^{r+1}}^2+\|\om^{in}\|_{X^{r}}^2\big)\\&\qquad+C\int_0^t\e^2\| P_{\geq N(\e)}(\pa_y,\e\pa_x)u^R\|_{X^{r+1}}^2+\|\pa_y\om^{in}\|_{X^{r}}^2+\|\om^{in}\|_{X^{r+\f{\s}{2}}}^2ds,
\\
&\int_0^t\|P_{\geq N(\e)}(\mathcal{N}_u,\e \mathcal{N}_v)\|_{X^{r+1-\f{\s}2}}^2ds\leq C\int_0^t \e^2\|P_{\geq N(\e)}(\pa_y,\e\pa_x)(u^R,\e v^R)\|_{X^{r+1}}^2+\|\om^{in}\|_{X^{r+\f{\s}{2}}}^2ds,
\end{align*}
for any $t\in[0,T]$.
\end{proposition}

\begin{proof}
By the definition of $\mathcal{N}$, we have
\begin{align*}
\int_0^t\|\mathcal{N}\|_{X^{r-\f{\s}4}}^2ds\leq \int_0^t\|u^R\pa_x \om^R\|_{X^{r-\f{\s}4}}^2ds+\int_0^t\|v^R\pa_y \om^R\|_{X^{r-\f{\s}4}}^2ds=I_1+I_2.
\end{align*}
It follows from Lemma \ref{lem:product-Gev} and \eqref{eq:uR-elliptic} that 
\begin{align*}
I_1\leq& \int_0^t \|\f{u^R_\Phi}{\e}\|_{L^\infty_y(H^{\f12+}_x)}^2\|\e\pa_x\om^R\|_{X^{r-\f{\s}4}}^2+\|u^R_\Phi\|_{L^\infty_y(H^{r-\f{\s}4}_x)}^2\|\pa_x\om^R\|_{X^{\f12+}}^2ds\\
\leq &\int_0^t\big(\|\f{u^R}{\e}\|_{X^{\f12+}}^2+\|\f{\pa_yu^R}{\e}\|_{X^{\f12+}}^2\big)\|\e\pa_x\om^R\|_{X^{r-\f{\s}4}}+\big(\|u^R\|_{X^{r-\f{\s}4}}^2+\|\pa_y u^R\|_{X^{r-\f{\s}4}}^2\big)\|\om^R\|_{X^{\f32+}}^2ds\\
\leq&\int_0^t \|\f{\om^R}{\e}\|_{X^{r-1}}^2\|\e\pa_x\om^R\|_{X^{r-\f{\s}4}}^2+\|\om^R\|_{X^{r-\f{\s}4}}^2\|\om^R\|_{X^{r-1}}^2ds\\
\leq& C\int_0^t\e^2\|\e\pa_x\om^R\|_{X^{r-\f{\s}4}}^2+\e^4\|\om^R\|_{X^{r-\f{\s}4}}^2ds,
\end{align*}
from which and the fact that 
\begin{align}\label{est: e^2pa_x om^R}
\|\e^2\pa_x\om^R\|_{X^{r-\f{\s}4}}\leq&\|\e^2P_{\geq N(\e)}\pa_x\om^R\|_{X^{r-\f{\s}4}}+\|\e^2{P_{\leq 2N(\e)}}\pa_x\om^R\|_{X^{r-\f{\s}4}}\\
\nonumber
\leq&C\e\| P_{\geq N(\e)}\e(\pa_y,\e\pa_x)(u^R,\e v^R)\|_{X^{r+1}}+C\|\om^R\|_{X^{r-1+\f{3\s}{4}}},
\end{align}
we infer that
\begin{align*}
I_1\leq C\int_0^t\e^2\| P_{\geq N(\e)}(\pa_y,\e\pa_x)(u^R,\e v^R)\|_{X^{r+1}}^2+\|\om^{in}\|_{X^{r}}^2ds.
\end{align*}

By Lemma \ref{lem:product-Gev} again, we get
\begin{align*}
I_2\leq& \int_0^t \int_0^t \|\f{v^R_\Phi}{\varphi}\|_{L^\infty_y(H^{\f12+}_x)}^2\|\varphi\pa_y\om^R\|_{X^{r-\f{\s}4}}^2+\|\e v^R_\Phi\|_{L^\infty_y(H^{r-\f{\s}4}_x)}^2\|\f{\pa_y\om^R}{\e}\|_{X^{\f12+}}^2ds\\
\leq&\int_0^t\|\om^R\|_{X^{\f32+}}^2\|\varphi\pa_y\om^R\|_{X^{r-\f{\s}4}}^2+\big(\|\e v^R\|_{X^{r-\f{\s}4}}^2+\|\e\pa_x u^R\|_{X^{r-\f{\s}4}}^2\big)\|\f{\pa_y\om^R}{\e}\|_{X^{\f12+}}^2ds\\
\leq& C\int_0^t\e^4\|\varphi\pa_y\om^R\|_{X^{r-\f{\s}4}}^2ds+\sup_{s\in[0,t]}\e^2\|\e\pa_x u^R\|_{X^{r-\f{\s}4}}^2.
\end{align*}
On the other hand, by Lemma \ref{lem:lift}, Lemma \ref{lem:h01-est} and \eqref{eq:uR-elliptic}, we have
\begin{align*}
\int_0^t\|\varphi\pa_y\om^R\|_{X^{r-\f{\s}4}}^2ds\leq& \int_0^t\|\pa_y\om^{in}\|_{X^{r-\f{\s}4}}^2+\|\varphi\pa_y\om^{bl}\|_{X^{r-\f{\s}4}}^2ds\\
\leq& C\int_0^t\|\pa_y\om^{in}\|_{X^{r}}^2+|(h^0,h^1)|_{X^{r+1-\s}}^2ds\\
\leq&C\int_0^t\|\pa_y\om^{in}\|_{X^{r}}^2+\|\om^{in}\|_{X^{r+\f{\s}{2}}}^2ds,
\end{align*}
and 
\begin{align*}
\|\e^2\pa_x u^R\|_{X^{r-\f{\s}4}}\leq & \|\e^2P_{\geq N(\e)}\pa_x u^R\|_{X^{r-\f{\s}4}}+\|\e^2{P_{\leq 2N(\e)}}\pa_x u^R\|_{X^{r-\f{\s}4}}\\
\leq&C\|P_{\geq N(\e)}\e u^R\|_{X^{r+1}}+\|\om^R\|_{X^{r-1+\f{\s}{4}}}\\
\le& C\|P_{\geq N(\e)}\e u^R\|_{X^{r+1}}+\|\om^{in}\|_{X^{r}}.
\end{align*}
This shows that 
\begin{align*}
I_2\leq C\e^4\int_0^t\|\pa_y\om^{in}\|_{X^{r}}^2+\|\om^{in}\|_{X^{r+\f{\s}{2}}}^2ds+C\sup_{s\in[0,t]}\big(\|P_{\geq N(\e)}\e u^R\|_{X^{r+1}}^2+\|\om^{in}\|_{X^{r}}^2\big).
\end{align*}

Now the first inequality follows from the estimates of $I_1$ and $I_2$.\smallskip

Next we estimate $\mathcal{N}_u$. Recalling that $\mathcal{N}_u=u^R\pa_x u^R+v^R\pa_y u^R$, we have
\begin{align*}
\int_0^t\|P_{\geq N(\e)}\mathcal{N}_u\|_{X^{r+1-\f{\s}2}}^2ds\leq& \int_0^t\|P_{\geq N(\e)}(u^R\pa_x u^R)\|_{X^{r+1-\f{\s}2}}^2ds+\int_0^t\|P_{\geq N(\e)}(v^R\pa_y u^R)\|_{X^{r+1-\f{\s}2}}^2ds\\
=& I_3+I_4.
\end{align*}
By Lemma \ref{lem:product-Gev} and \eqref{eq:uR-elliptic}, we have
\begin{align*}
I_3\leq& \int_0^t \|u^R_\Phi\|_{L^\infty_y(H^{\f12+}_x)}^2\|P_{\geq \f{N(\e)}{2}}\pa_x u^R\|_{X^{r+1-\f{\s}2}}^2+\|P_{\geq \f{N(\e)}{2}} u^R\|_{X^{r+1-\f{\s}2}}^2\|\pa_xu^R_{\Phi}\|_{L^\infty_y(H^{\f12+}_x)}^2ds\\
\leq& C\int_0^t \|\om^R\|_{X^{\f32+}}^2\|P_{\geq \f{N(\e)}{2}}\pa_x u^R\|_{X^{r+1-\f{\s}2}}^2ds\leq C\int_0^t \e^4\|P_{\geq \f{N(\e)}{2}}\pa_x u^R\|_{X^{r+1-\f{\s}2}}^2ds\\
\leq& C\int_0^t \e^2\|P_{\geq N(\e)}\e\pa_x u^R\|_{X^{r+1}}^2+\|P_{\geq\f{N(\e)}2}P_{\leq N(\e)}\e\om^R\|^2_{X^{r+1-\f{\s}2}}ds\\
\leq& C\int_0^t \e^2\|P_{\geq N(\e)}\e\pa_x u^R\|_{X^{r+1}}^2+\|\om^R\|_{X^{r}}^2ds\\
\leq&{C\int_0^t \e^2}\|P_{\geq N(\e)}\e\pa_x u^R\|_{X^{r+1}}^2+\|\om^{in}\|_{X^{r+\f{\s}{2}}}^2ds.
\end{align*}
Similarly, we have
\begin{align*}
I_4\leq&C\int_0^t \|v^R_\Phi\|_{L^\infty_y(H^{\f12+}_x)}^2\|P_{\geq\f{N(\e)}{2}}\pa_y u^R\|_{X^{r+1-\f{\s}2}}^2+\|P_{\geq\f{N(\e)}{2}}v^R_\Phi\|_{L^\infty_y(H^{r+1-\f{\s}2}_x)}^2\|\pa_y u^R\|_{X^{\f12+}}^2ds\\
\leq& C\int_0^t \|\om^R\|_{X^{\f32+}}^2\big(\|P_{\geq\f{N(\e)}{2}}\pa_y u^R\|_{X^{r+1-\f{\s}2}}^2+\|P_{\geq\f{N(\e)}{2}}\pa_xu^R\|_{X^{r+1-\f{\s}2}}^2\big)ds\\
\leq& C\int_0^t \|P_{\geq N(\e)}\e^2(\pa_y,\pa_x) u^R\|_{X^{r+1}}^2+\|\e^2P_{\geq \f{N(\e)}{2}}P_{\leq N(\e)}(\pa_y u^R,\pa_x u^R)\|_{X^{r+1-\f{\s}2}}^2ds\\
\leq &{C\int_0^t \e^2}\|P_{\geq N(\e)}(\pa_y,\e\pa_x) u^R\|_{X^{r+1}}^2+\|\e P_{\geq \f{N(\e)}{2}}P_{\leq N(\e)}\om^R\|_{X^{r+1-\f{\s}2}}^2ds\\
\leq &{C\int_0^t \e^2}\|P_{\geq N(\e)}(\pa_y,\e\pa_x) u^R\|_{X^{r+1}}^2+\|\om^{in}\|_{X^{r+\f{\s}{2}}}^2ds.
\end{align*}
This shows that 
\begin{align*}
\int_0^t\|P_{\geq N(\e)}\mathcal{N}_u\|_{X^{r+1-\f{\s}2}}^2ds\leq& {C\int_0^t \e^2}\|P_{\geq N(\e)}(\pa_y,\e\pa_x) u^R\|_{X^{r+1}}^2+\|\om^{in}\|_{X^{r+\f{\s}{2}}}^2ds.
\end{align*}

The estimate for $\e\mathcal{N}_v$ is obtained by changing $u^R$ into $\e v^R$ and we omit details. 
\end{proof}

\section{Proof of Theorem \ref{thm:limit}}
This section is devoted to proving Theorem \ref{thm:limit}. \smallskip

\begin{itemize}

\item {\bf Local well-posedness}.  The local well-posedness of the anisotropic Navier–Stokes equations in the Gevrey class can be proved by a standard energy method. Here we omit the details.
Let $T_1$ be the maximal existence time of the solution.

\item {\bf Bootstrap assumption}:
\begin{align*}
\sup_{s\in[0,t]}\|\om^R(s)\|_{X^{r-1}}^2+\int_0^t\|(\pa_y \om^R,\e\pa_x\om^R)\|_{X^{r-1}}^2\leq \mathfrak{C}\e^4
\end{align*}
for  any $t\in[0,T_1].$ Here $\mathfrak{C}$ is determined later.

\item {\bf Energy functional}:
\begin{align*}
E(t)=&\|\om^{in}(t)\|_{X^r}^2+A\e^2\|P_{\geq N(\e)}(u^{R},\e v^R)(t)\|_{X^{r+1}}^2+\|P_{\geq N(\e)}(u^{R},\e v^R)(t)\|_{X^{r+1-\s}}^2,\\
G(t)=&\|\om^{in}(t)\|_{X^{r+\f{\s}{2}}}^2+A\e^2\|P_{\geq N(\e)}(u^{R},\e v^R)(t)\|_{X^{r+1+\f{\s}{2}}}^2+\|P_{\geq N(\e)}(u^{R},\e v^R)(t)\|_{X^{r+1-\f{\s}{2}}}^2,\\
D(t)=&\|(\pa_y,\e\pa_x) \om^{in}(t)\|_{X^r}^2+A\e^2\|P_{\geq N(\e)}(\pa_y,\e\pa_x)( u^{R},\e v^R)\|_{X^{r+1}}^2\\
&+\|P_{\geq N(\e)}(\pa_y,\e\pa_x)( u^{R},\e v^R)\|_{X^{r+1-\s}}^2,
\end{align*}
where  $A$ is a large constant determined later.

\item {\bf Energy estimates}. It follows from Proposition \ref{prop:om-in},  Proposition \ref{prop:ur-high} 
and Proposition \ref{prop:non} that 
\begin{align*}
&\sup_{s\in[0,t]}E(s)+\beta\int_0^tG(s)ds+\int_0^t D(s)ds\\
&\leq C\int_0^tG(s)ds+\big(\f{C}{A}+C_1\d\big)\int_0^t D(s)ds+Ct\e^4+C_1\delta \sup_{s\in[0,t]}E(s).
\end{align*}
Here $C_1$ is independent of $\delta$.  Taking $\beta$ and $A$ large enough and $\d$  small enough, we obtain
\begin{align}
\sup_{s\in[0,t]}E(s)+\beta\int_0^tG(s)ds+\int_0^t D(s)ds\leq Ct\e^4\label{eq:energy}
\end{align}
for any  $t\in[0,T_1].$

\item {\bf Improving the bootstrap assumption.} It follows from Corollary \ref{cor:omR}  and  \eqref{eq:energy} that 
\begin{align*}
\sup_{s\in[0,t]}\|\om^R(s)\|_{X^{r-1}}^2+\int_0^t\|(\pa_y \om^R,\e\pa_x\om^R)\|_{X^{r-1}}^2\leq Ct\e^4\leq \f{\mathfrak{C}}{2}\e^4.
\end{align*}
by choosing $\mathfrak{C}$ so that $\mathfrak{C}\geq 2CT.$ This in particular implies that $T_1\ge T$.

\item {\bf Stability in $L^2\cap L^\infty$}. By the Sobolev  embedding, we get
\begin{align*}
\|(u^R,\e v^R)\|_{L^2_{x,y}\cap L^\infty_{x,y}}\leq C\e^2.
\end{align*}

\end{itemize}

\section* {Acknowledgments}
C. Wang is partially supported by NSF of China under Grant 11701016. Y. Wang is partially supported by China Postdoctoral Science Foundation 8206200009.

\end{CJK*}

\end{document}